\documentclass[a4paper]{amsart}
\usepackage{amssymb, enumitem}
\usepackage{tikz}
\usepackage[all]{xy}
\def\firstcircle{(70:0.75cm) circle (1cm)}
\def\secondcircle{(190:0.75cm) circle (1cm)}

\usepackage{hyperref, aliascnt}
\usepackage{mathtools}
\usepackage{xcolor}
\usepackage{bbm}
\def\today{\number\day\space\ifcase\month\or   January\or February\or
   March\or April\or May\or June\or   July\or August\or September\or
   October\or November\or December\fi\   \number\year}

\theoremstyle{definition}
\newtheorem{lma}{Lemma}[section]

\newaliascnt{thmCt}{lma}
\newtheorem{thm}[thmCt]{Theorem}
\aliascntresetthe{thmCt}

\newaliascnt{corCt}{lma}
\newtheorem{cor}[corCt]{Corollary}
\aliascntresetthe{corCt}

\newaliascnt{cnjCt}{lma}

\aliascntresetthe{cnjCt}

\newaliascnt{propCt}{lma}
\newtheorem{prop}[propCt]{Proposition}
\aliascntresetthe{propCt}

\newtheorem*{thm*}{Theorem}
\newtheorem*{cor*}{Corollary}
\newtheorem*{prop*}{Proposition}

\newcounter{theoremintro}
\newtheorem{defintro}[theoremintro]{Definition}
\newtheorem{thmintro}[theoremintro]{Theorem}

\newaliascnt{pgrCt}{lma}

\aliascntresetthe{pgrCt}

\newaliascnt{dfCt}{lma}
\newtheorem{df}[dfCt]{Definition}
\aliascntresetthe{dfCt}

\newaliascnt{remCt}{lma}
\newtheorem{rem}[remCt]{Remark}
\aliascntresetthe{remCt}

\newaliascnt{remsCt}{lma}

\aliascntresetthe{remsCt}

\newaliascnt{egCt}{lma}
\newtheorem{eg}[egCt]{Example}
\aliascntresetthe{egCt}

\newaliascnt{egsCt}{lma}

\aliascntresetthe{egsCt}

\newaliascnt{qstCt}{lma}

\aliascntresetthe{qstCt}

\newaliascnt{pbmCt}{lma}

\aliascntresetthe{pbmCt}

\newaliascnt{notaCt}{lma}
\newtheorem{nota}[notaCt]{Notation}
\aliascntresetthe{notaCt}

\newcommand{\beq}{\begin{equation}}
\newcommand{\eeq}{\end{equation}}
\newcommand{\beqa}{\begin{eqnarray*}}
\newcommand{\eeqa}{\end{eqnarray*}}
\newcommand{\bal}{\begin{align*}}
\newcommand{\eal}{\end{align*}}
\newcommand{\bi}{\begin{itemize}}
\newcommand{\ei}{\end{itemize}}
\newcommand{\be}{\begin{enumerate}}
\newcommand{\ee}{\end{enumerate}}

\newcommand{\Z}{{\mathbb{Z}}}
\newcommand{\C}{{\mathbb{C}}}
\newcommand{\N}{{\mathbb{N}}}

\newcommand{\id}{{\mathrm{id}}}

\newcommand{\Aut}{{\mathrm{Aut}}}

\newcommand{\ca}{$C^*$-algebra}

\newcommand{\I}{\infty}

\title[]{Decomposable partial actions}

\thanks{
The second named author was partially supported
by the Deutsche Forschungsgemeinschaft's (DFG) eigene Stelle, and by a Postdoctoral Research Fellowship
from the Humboldt Foundation. The third named author was supported by a Kreitman Foundation Fellowship, a Minerva
Fellowship Programme, and by an Israel Science Foundation grant no.~476/16. The second and third named authors were partially supported by the 
DFG through SFB 878 and under Germany's Excellence Strategy – EXC 2044 – 390685587, Mathematics M\"unster – Dynamics – Geometry – Structure, and by ERC Advanced Grant 
834267 - AMAREC}

\author[Fernando Abadie]{Fernando Abadie}
\address{Fernando Abadie
Centro de Matem\'atica, Facultad de Ciencias,
Universidad de la Rep\'ublica, Igu\'a 4225, 11400
Montevideo, Uruguay.}
\email{fabadie@cmat.edu.uy}
\urladdr{http://www.cmat.edu.uy/~fabadie/}

\author[Eusebio Gardella]{Eusebio Gardella}
\address{Eusebio Gardella
Mathematisches Institut, Fachbereich Mathematik und Informatik der
Universit\"at M\"unster, Einsteinstrasse 62, 48149 M\"unster, Germany.}
\email{gardella@uni-muenster.de}
\urladdr{www.math.uni-muenster.de/u/gardella/}

\author[Shirly Geffen]{Shirly Geffen}
\address{Shirly Geffen
Department of Mathematics, Ben-Gurion University of the Negev, Be’er
Sheva 8410501, Israel.}
\email{shirlyg@post.bgu.ac.il}
\urladdr{https://shirlygeffen.com/}
\begin{document}

\begin{abstract}
We define the decomposition property for partial actions of discrete groups on $C^*$-algebras.
Decomposable partial systems
appear naturally in practice, and
many commonly occurring partial actions can be decomposed into partial actions
with the decomposition property. For instance, any partial action of a finite 
group is an iterated extension of decomposable systems. 

Partial actions with the decomposition property are always globalizable and ame\-nable, regardless of the acting group, and their 
globalization can be explicitly described in terms of certain global sub-systems.
A direct computation of their crossed products is also carried out.
We show that partial actions with the decomposition property behave in many ways
like global actions of finite groups (even when the acting group is infinite),
which makes their study particularly accessible. For example, there exists
a canonical faithful conditional expectation onto the fixed point algebra, which
is moreover a corner in the crossed product in a natural way. (Both of these
facts are in general false for partial actions of finite groups.) 
As an application, we show that freeness of a topological partial action with 
the decomposition property is equivalent to its fixed point algebra being Morita
equivalent to its crossed product. We also show by example that this fails for
general partial actions of finite groups.
\end{abstract}

\maketitle

\vspace{-.5cm}

\renewcommand*{\thetheoremintro}{\Alph{theoremintro}}
\section{Introduction}
The computation of the $K$-theory of \ca s is a challenging task, 
whose origins date back to Cuntz's seminal work \cite{Cun_theory_1981}.
Obtaining descriptions of the $K$-groups is 
useful when trying to distinguish \ca s, and in many cases of interest, 
algebras with isomorphic $K$-theory are themselves isomorphic (see \cite{EllTom_regularity_2008} for a survey on this topic).
There are a number of tools that help us compute the $K$-groups of a given 
\ca, and the Pimsner-Voiculescu six-term exact sequence for crossed products by
$\Z$-actions \cite{PimVoi_1980} is arguably one of the most powerful ones. 
A classical result of Landstad provides an abstract characterization of 
$\Z$-crossed products as those algebras carrying a circle action with 
suitably large spectral subspaces, thus granting access to the 
Pimsner-Voiculescu exact sequence without knowing beforehand that the algebra
in question is a $\Z$-crossed product.
However, many naturally occurring circle
actions tend to have rather ``small'' spectral subspaces (this is typically
the case for gauge actions), and Landstad's result therefore does not apply.
This problem was tackled by Exel in~\cite{Exe_circle_1994}, where he showed
that under very mild assumptions on the circle action, such algebras are
isomorphic to the crossed product of what he called a \emph{partial 
automorphism}, that is, an isomorphism between two ideals. This, in combination
with a version of Pimsner-Voiculescu for partial automorphisms also obtained 
in \cite{Exe_circle_1994}, allows one to compute the $K$-groups of a large
class of \ca s, including all Cuntz-Krieger algebras.

The notion of a partial action of a general discrete group $G$ on a \ca\ $A$
was introduced by McClanahan in \cite{McC_theory_1995}:
it is a collection $(A_g)_{g\in G}$ of ideals of $A$ and isomorphisms $\alpha_g\colon A_{g^{-1}}\to A_g$, for $g\in G$, such that $\alpha_1=\id_A$ and $\alpha_{gh}$ extends $\alpha_g\circ \alpha_h$ wherever the decomposition is well-defined.
By taking all domains to be $A$ we recover the usual definition of a 
group action (which we call a \textit{global} action, to avoid confusion).
The generality of these objects is 
illustrated by the fact that large classes of \ca s
can be described as partial crossed products of 
commutative \ca s, including 
AF-algebras \cite{Exe_approximately_1995}, Bunce-Deddens algebras \cite{Exe_bunce_1994}, Exel-Laca algebras \cite{ExeLac_cuntz_1999},
or the Jiang-Su algebra \cite{DeePutStr_constructing_2018}, thus 
considerably enlarging the toolkit available to study them. 

The study of partial actions unveiled a number of unexpected connections, particularly as partial actions were realized to be closely related 
to inverse semigroup actions. Indeed,
Exel showed in \cite{Exe_partial_1998} that to every discrete group $G$ one can canonically
associate a semigroup $E(G)$, sometimes referred to as the \emph{Exel semigroup} 
of $G$, in such a way that partial actions of $G$ on 
a topological space $X$ are in natural one-to-one correspondence with 
inverse semigroup morphisms $E(G)\to \mathrm{Homeo}_{\mathrm{par}}(X)$
from $E(G)$ to the inverse semigroup of partial homeomorphisms of $X$. Moreover,
the inverse semigroup \ca\ $C^*(E(G))$ can be canonically identified with the 
partial group algebra $C^*_{\mathrm{par}}(G)$ of $G$. For further connections 
to inverse semigroups and related objects, we refer the reader to 
Exel-Steinberg's Proceedings of the ICM \cite{ExeSte_inverse_2018}.
Connections to groupoids have been explored
in \cite{Aba_partial_2004, Aba_partial_2005}.

Despite the numerous advances within the theory of partial dynamical systems, 
a number of aspects remain unexplored, and many others exhibit behaviors 
that differ dramatically from the case of global actions. This is in spite
of the tight connections with global systems discovered in \cite{Aba_enveloping_2003} in the context of the globalization problem. 
For example,
in the case of finite groups, and even for actions that admit a globalization, virtually all averaging arguments (as well as 
their consequences) that are standard for global actions, completely break 
down in the partial setting. Among others, the fixed point algebra may be very
small; there is in general no conditional expectation onto it; and it is in general
not a corner in the crossed product. The lack of approximate identities that
are compatible with the partial action is also a source of difficulties in this setting. 

The present work originates in the attempts by the authors to obtain a better and
more systematic 
understanding of the internal structure of partial actions.
To illustrate this, let us consider first the smallest nontrivial finite 
group $\Z_2$, and let $\alpha$
be a partial action of $\Z_2$ on a compact Hausdorff space $X$.
This amounts to a choice of an open subset $U\subseteq X$ and an order two 
homeomorphism $\sigma$ of it. The restriction of this partial system to $U$
is a global system, while the 
``remainder'' $Y=X\setminus U$ is acted upon trivially. In other words, there
is an equivariant topological extension
$\xymatrix{U \ \ar@{^{(}->}[r]& X \ar@{->>}[r] & Y}$. 
The equivariant structure of $U$ and $Y$ is completely understood, 
and the complexity of the partial action $\alpha$ is determined by the way 
how $U$ and $Y$ are glued together. 
Regardless of how complicated this gluing is, 
many aspects of $C(X)\rtimes_\alpha \Z_2$ can be deduced
from the fact that it is an extension of $C(Y)$ by 
$C_0(U)\rtimes_\sigma \Z_2$, the latter being a global crossed product.

For larger groups, one has to iterate the process of decomposing the action
into simpler sub-systems. For example, an action of $\Z_3$ on a compact Hausdorff
space $X$ is given by the choice of two open subsets $U_1, U_2\subseteq X$, 
and a homeomorphism $\sigma_1\colon U_2\to U_1$ satisfying $\sigma_1^3=\mathrm{id}$ wherever
the composition is well-defined.
The restriction of this partial system to $U=U_1\cap U_2$ is thus global, 
and we get an extension
\[\xymatrix{0\ar[r] & C_0(U)\ar[r] &C(X) \ar[r] & C(Y)\ar[r] &0,}\]
where $Y=X\setminus U$. In this case, however, the induced partial action 
on $Y$ is not trivial. Instead, we may decompose it further: with 
$V_1=U_1\setminus U_2$ and $V_2=U_2\setminus U_1$, the homeomorphisms induced
by $\sigma_1$ and $\sigma_2$ ``exchange'' $V_1$ and $V_2$ (and act internally
in a global manner),
while the complement $Z=Y\setminus (V_1\sqcup V_2)$ 
carries the trivial partial action.   	
   	Thus, we get an equivariant extension
\[\xymatrix{0\ar[r] & C_0(V_1\sqcup V_2)\ar[r] &C(Y) \ar[r] & C(Z)\ar[r] &0,}\]
where now the action on $V_1\sqcup V_2$ is a combination of a translation and 
an order-3 homeomorphism, and the action on $Z$ is trivial. We have therefore
reduced the understanding of $\Z_3\curvearrowright X$ to the understanding of 
the systems $\Z_3\curvearrowright U$ and $\Z_3\curvearrowright V_1\sqcup V_2$,
together with the understanding of the above equivariant extensions.
\vspace{-.15cm}
\def\firstcircle{(0:0.7cm) circle (1.2cm)}
\def\secondcircle{(180:0.8cm) circle (1.2cm)}
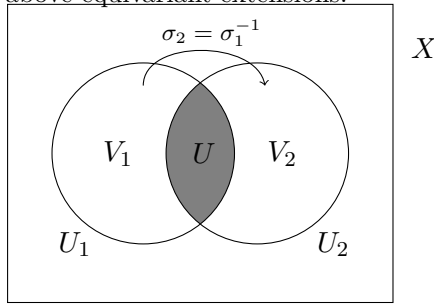
    \begin{center}
    	\begin{tikzpicture}
   		\begin{scope}
     	\clip \secondcircle;
   		\fill[gray] \firstcircle;
   		\end{scope}
   		\begin{scope}
   		\clip \secondcircle;
   		\end{scope}
   		\begin{scope}
   		\clip \firstcircle;
   		\end{scope}
   		\begin{scope}
   		\clip \firstcircle;
   		\clip \secondcircle;
   		\end{scope}[font=\large]
   		\draw \firstcircle node[right] {$V_2$};
   		\draw \secondcircle node [left] {$V_1$};
   		\draw ([xshift=-1.5em,yshift=-2em]current bounding box.south west)
    	rectangle
    	([xshift=1.5em,yshift=2em]current bounding box.north east);
   		\node at (0.1,1.6) [font=\small]{$\sigma_2=\sigma_1^{-1}$};
		\node at (2.9,1.4){$X$};
   		\node [font=\large] {$U$};
   		\node at (-1.7,-1.2) {$U_1$};
   		\node at (1.7,-1.2) {$U_2$};
   		\path[->] (-0.8,0.9) edge [out=90, in=100] (0.8,0.9);
   		\end{tikzpicture}
   	\end{center} 

For general finite groups, the decomposition process is rather subtle: one 
obtains a larger number of extensions, and the structure of 
the intermediate $G$-equivariant ideals becomes very complicated. 
Problems of this nature already appeared in the work of the 
third-named author \cite{Gef_nuclear_2020}.
In particular, for such a decomposition to be useful in practice, one
would like to be able to compute the crossed products of the intermediate
systems. 

Our attempts to shed light over the problem described above
led us to isolate and study a property that the intermediate systems
satisfy, which we call the \emph{decomposition property}. 
For a discrete group $G$ and $n\in\N$, we set 
\[\mathcal{T}_n(G)=\{\tau\subseteq G\colon 1\in \tau, |\tau|=n\},\]
endowed with the partial action of $G$ by left multiplication. 
For a partial action $\alpha=((A_g)_{g\in G}, (\alpha_g)_{g\in G})$
on a \ca\ $A$, we set $A_\tau=\bigcap_{g\in\tau}A_g$ for $\tau\in\mathcal{T}_n(G)$. 

\begin{defintro}\label{defintro:IntProp}
Let $n\in\mathbb{N}$ and let $\alpha$ be a partial action of $G$ on $A$. We say that $\alpha$ has the \textit{$n$-decomposition property} if 
	\begin{enumerate}
		\item[(a)] $A=\overline{\sum\limits_{\tau\in \mathcal{T}_n(G)} A_\tau}$, and
		\item[(b)] $A_\tau\cap A_g=\{0\}$ for all $\tau\in \mathcal{T}_n(G)$ and all $g\notin \tau$.
	\end{enumerate}
We say that $\alpha$ is \textit{decomposable} $\alpha$ has the $n$-decomposition property for some $n\in\N$.
\end{defintro}

Decomposable partial 
actions form a rich class of systems
whose essential structure is roughly a combination of 
a global action of a finite group (even if the original group is infinite), 
a translation, and a trivial action.
In this sense, actions with the decomposition property are 
much more accessible than general partial actions, the upshot being that many naturally
occurring partial actions can be approximated by systems with the decomposition
property. This is particularly the case for arbitrary 
partial actions of finite groups.
This fact is extensively exploited in \cite{AbaGarGef_partial_2020}.

Let $\alpha$ be a partial action with the $n$-decomposition property. Then 
$\alpha$ decomposes as a direct sum of restrictions to sub-systems of the 
form $A_{G\cdot \tau}=\bigoplus_{g\in \tau} A_{g^{-1}\tau}$. Thus, it 
suffices to understand these in order to understand $\alpha$. 
For a given $\tau$, we set $H_\tau=\{h\in G\colon h\tau=\tau\}$,
which is a finite subgroup of $G$ that acts globally on $A_\tau$. Moreover, 
there are $x_0,x_1,\ldots,x_{m_\tau}\in G$ such that 
$\tau=\bigsqcup_{j=0}^{m_\tau} H_\tau x_j$. This data determines most of the relevant
features of the action on $A_{G\cdot \tau}$ (and hence on $A$), 
which we study in detail in this work. 

The main result of Section~3 (\autoref{thm:InterPropGlobz}) 
asserts that actions with the decomposition property 
are always amenable and globalizable, and provides an explicit description of the globalization in terms of induced systems 
(see \autoref{df:Ind}):

\begin{thmintro}\label{thmintro:EnvAct}
Let $\alpha$ be a partial action of $G$ on a $C^*$-algebra $A$ with the $n$-decomposition property. Then each $G\curvearrowright A_{G\cdot \tau}$ is amenable and globalizable, and its 
globalization is 
$G\curvearrowright \mathrm{Ind}_{H_\tau}^{G}(A_\tau)$. It follows that $\alpha$ is
amenable and globalizable, 
and globalization is a direct sum of actions of the form 
$G\curvearrowright \mathrm{Ind}_{H_\tau}^{G}(A_\tau)$.
\end{thmintro}

Combining the above theorem with results of the first-named
author from \cite{Aba_enveloping_2003}, it follows that the partial crossed
product $A\rtimes_\alpha G$ can be computed, up to Morita equivalence,
using the crossed products of the global sub-systems $H_\tau\curvearrowright A_{\tau}$. However, for some applications it is necessary to have a 
computation of $A\rtimes_\alpha G$ up to isomorphism and not just up to
Morita equivalence, and we do this in \autoref{prop:CrossedProductTupleOrb}:

\begin{thmintro}\label{thmintro:CrossProd}
Let $\alpha$ be a partial action of $G$ on a $C^*$-algebra $A$ with the $n$-decomposition property. 	
Then $A_{G\cdot\tau}\rtimes_\alpha G\cong M_{m_\tau +1}(A_\tau\rtimes H_\tau)$. In particular, 
$A\rtimes_\alpha G$ is a direct sum of algebras of the form 
$M_{m_\tau +1}(A_\tau\rtimes H_\tau)$, for $\tau\in\mathcal{T}_n(G)$.
\end{thmintro}

Using the above result, we provide 
an explicit computation of the partial group algebra $C^*_{\mathrm{par}}(G)$ of a finite group $G$ from \cite{Exe_partial_1998}; see \autoref{thm:PartialGpAlg}. Indeed, $C^*_{\mathrm{par}}(G)$ can be 
realized as the crossed product of the partial action of $G$ on
$\bigsqcup_{n=1}^{|G|} \mathcal{T}_n(G)$. Moreover, 
$G\curvearrowright \mathcal{T}_n(G)$ has the $n$-decomposition property, so 
its crossed product can be computed using Theorem~\ref{thmintro:CrossProd}; see \autoref{thm:PartialGpAlg}. 
An equivalent description of $C^*_{\mathrm{par}}(G)$ was obtained in \cite{DokExePic_partial_2000}.

Perhaps some of the most surprising features of actions with the decomposition
property refer to their $G$-invariant elements and fixed point subalgebras. 
For global actions, the theory works best in the setting of finite groups: for instance, there is a faithful conditional expectation 
$E\colon A\to A^G$, and there is
an injective homomorphism 
$c\colon A^G\to A\rtimes G$ whose image is a corner in $A\rtimes G$. 
Both of these maps fail to exist for 
infinite groups. For partial actions, even of finite groups, additional 
complications arise and these maps rarely exist; 
see \autoref{eg:NoCondExp}. The following 
(\autoref{prop:CondExp}, \autoref{prop:CornerMap}) is
thus unexpected:

\begin{thmintro}\label{thmintro:fixPt}
Let $\alpha$ be a decomposable partial action of $G$ on a $C^*$-algebra $A$. Then there are a faithful conditional expectation 
$E\colon A\to A^G$ and an injective homomorphism 
$c\colon A^G\to A\rtimes G$ whose image is a corner in $A\rtimes G$. 
\end{thmintro}

Finally, in \autoref{thm:Freeness} we give a characterization of 
freeness for topological partial actions 
with the decomposition property, in terms of the corner map $c$:

\begin{thmintro}
Let $\sigma$ be a decomposable partial action of $G$ on a locally compact
Hausdorff space $X$. Then $\sigma$ is free
if and only if the image of $c\colon C_0(X)^G\to C_0(X)\rtimes_\sigma G$
is full.
\end{thmintro}

In Theorem~E, we could have used $C_0(X/G)$ in place of $C_0(X)^G$, since the two agree for a decomposable action. 
For general partial actions, even of finite groups, the above characterization fails (even if one considers $C_0(X/G)$); see \autoref{eg:MoritaEqFails}.

\section{The decomposition property}\label{sec: decomposition property}

In this section, we define the decomposition property of a partial action and 
study some of its basic features.
We fix a discrete group $G$ and $n\in\N$.

\begin{df}\label{df:TnG}
Given $n\in\N$, we define the \emph{space of $n$-tuples} of $G$ to be 
\[\mathcal{T}_n(G)=\{\tau \subseteq G \colon 1\in \tau \mbox{ and } |\tau|=n\}.\]
(Note that $\mathcal{T}_n(G)=\emptyset$ whenever $G$ is finite and $n>|G|$.) 
For $g\in G$, we set 
$\mathcal{T}_n(G)_{g}=\{\tau\in \mathcal{T}_n(G)\colon g\in \tau\}$.
There is a canonical partial action 
$\texttt{Lt}$ of $G$ on $\mathcal{T}_n(G)$ with 
$\texttt{Lt}_g\colon \mathcal{T}_n(G)_{g^{-1}}\to \mathcal{T}_n(G)_g$ given by 
$\texttt{Lt}_g(\tau)=g\tau$ for all $\tau\in\mathcal{T}_n(G)_{g^{-1}}$.
\end{df}

We regard $\mathcal{T}_n(G)$ as a topological space, equipped with the discrete topology.

\begin{nota}
Let 
$\alpha=((A_g)_{g\in G}, (\alpha_g)_{g\in G})$ be a partial action of a discrete group $G$ on a \ca\ $A$, and let $n\in\N$.
For $\tau\in \mathcal{T}_n(G)$, we write $A_{\tau}$ for the ideal $A_{\tau}=\bigcap_{g\in \tau} A_{g}$.
Observe that for $g\in G$ and $\tau\in \mathcal{T}_n(G)_{g^{-1}}$, we have
$\alpha_{g}(A_{\tau})=A_{g\tau}$.
For $\tau \in \mathcal{T}_n(G)$, we write $G\cdot \tau\subseteq \mathcal{T}_n(G)$ for the orbit of
$\tau$ (with respect to the partial action $\texttt{Lt}$ from \autoref{df:TnG}), and we set $A_{G\cdot\tau}=\sum_{g\in \tau^{-1}} A_{g\tau}$.
\end{nota}

The following is the main technical definition of this work.

\begin{df}\label{df:nIntProp}
Let $G$ be a discrete group and let 
$\alpha=((A_g)_{g\in G}, (\alpha_g)_{g\in G})$ be a partial action on a \ca\ $A$. 
For $n\in \N$, we say that $\alpha$ has the \emph{$n$-decomposition property} if
\be\item[(a)] $A=\overline{\sum\limits_{\tau\in\mathcal{T}_n(G)}A_\tau}$, and 
\item[(b)]$A_{\tau}\cap A_g= \{0\}$ for all $\tau\in\mathcal{T}_{n}(G)$ and all $g\in G$ such that $g\notin \tau$. \ee
We say that $\alpha$ has the \emph{decomposition property} if it has the $n$-decomposition property for some $n\in \N$. 
A partial action on a locally compact space $X$ is said to have
the \emph{($n$-)decomposition property} if its induced partial action on $C_0(X)$
has it.  
\end{df} 

The following easy observation will be needed later.

\begin{rem} \label{rem:DisjTuples}
In the context of \autoref{df:nIntProp}, 
condition~(b) implies that $A_\tau\cap A_\sigma=\{0\}$ whenever
$\tau,\sigma\in\mathcal{T}_n(G)$ are distinct. In particular, 
we think
of $A$ as a direct sum of the orthogonal ideals appearing in
condition~(a). 
In fact if $n>1$, then $\alpha$ has the
$n$-decomposition property if and only if $A$ is
isomorphic to the $C^*$-algebraic direct sum
$\bigoplus_{\tau\in\mathcal{T}_n(G)}A_\tau$, in such a way
that $A_g$ corresponds to $\bigoplus_{\tau\in\mathcal{T}_n(G)_g}A_\tau$. Indeed, assume that
the latter holds, and let
$\tau\in\mathcal{T}_n(G)$ and $g\in G\setminus\tau$. 
Pick any $g'\in \tau\setminus\{1\}$ and set
$\sigma=(\tau\cup\{g\})\setminus\{g'\}$. Then
$0=A_{\sigma}A_\tau=A_{g}A_\tau$, so condition~(b) holds.
\end{rem}

We now turn to examples. 
The extreme cases are easy to describe:

\begin{eg}\label{eg:IntProp}
Let $\alpha$ be a partial action of a discrete group $G$
on a \ca\ $A$.
\be
\item $\alpha$ has the 1-decomposition property if and only if $A_g=\{0\}$ for all $g\in G\setminus\{1\}$.
This is the \emph{trivial partial action} of $G$ on $A$.
\item If $G$ is finite, then $\alpha$ has the $|G|$-decomposition property if and only if $\alpha$ is global. 
\ee
\end{eg}

Next, we show that the canonical partial action of $G$ on $\mathcal{T}_n(G)$ has the $n$-decomposition property. This is the \emph{prototypical} 
partial action with the $n$-decomposition property. 

\begin{prop}\label{prop:TnGhasInterProp}
Let $G$ be a discrete group and let $n\in\N$. 
Then the partial action $\texttt{Lt}$ 
of $G$ on $\mathcal{T}_n(G)$ described in \autoref{df:TnG} 
has the $n$-decomposition property.
\end{prop}
\begin{proof}
For $\tau\in\mathcal{T}_n(G)$, it is easy to check that $C_0(\mathcal{T}_n(G))_\tau=C(\{\tau\})$, so condition (b) in \autoref{df:nIntProp} is satisfied. 
Since $\mathcal{T}_n(G)$ is discrete, we have 
$C_0(\mathcal{T}_n(G))\cong
\bigoplus_{\tau\in\mathcal{T}_n(G)}C(\{\tau\})$, 
so condition (a) is also satisfied.
\end{proof}

Decomposable partial actions form a rich class whose structure is, in a rough sense, a combination of
the trivial action, an action by translation, and a global action of a finite group. 
In this sense, decomposable partial actions are 
much more accessible than general partial actions. The upshot of this approach is the fact that many naturally
occurring partial actions can be written either as iterated extensions or limits of decomposable partial actions; this is in particular true for partial actions of finite groups by \autoref{thm:ExtensionIntersProp}.


The following lemma will be fundamental in most of our analysis.

\begin{lma}\label{lma:OrbitTuples}
Let $\tau\in\mathcal{T}_n(G)$ and set $H_\tau=\{h\in G\colon h\tau=\tau\}$.
Then
\be\item $H_\tau$ is finite a subgroup of $G$, and $|H_\tau|$ divides $n$;
\item With $m_\tau=\frac{n}{|H_\tau|}-1$, there exist 
$x^\tau_1,\ldots,x^\tau_{m_\tau}\in G$ distinct such that
\[\tau=H_\tau\sqcup H_\tau x^\tau_1 \sqcup\ldots\sqcup H_\tau x^\tau_{m_\tau}.\]
\item If $y^\tau_1,\ldots,y^\tau_{m_\tau}\in G$ satisfy $\tau=H_\tau\sqcup H_\tau y^\tau_1 \sqcup\ldots\sqcup H_\tau y^\tau_{m_\tau}$,
then there exist a permutation $\sigma\in S_{m_\tau}$ and $h_1,\ldots,h_{m_\tau}\in H_\tau$ such that $y_j=h_jx_{\sigma(j)}$ for all $j$.
\ee
\end{lma}
\begin{proof}
It is clear that $H_\tau$ is a subgroup of $G$. We will prove items (1) and (2) simultaneously. 
The condition $h\tau=\tau$ for all $h\in H_\tau$ implies that $\tau$ is 
$H_\tau$-invariant, 
and that $H_\tau\subseteq \tau$, since $1\in \tau$. 
Hence $H_\tau$ is finite.
$H_\tau$ acts globally on the finite set $\tau$. Thus $\tau$ is a disjoint union of $H_\tau$ orbits, $\tau=H_\tau x^\tau_0 \sqcup H_\tau x^\tau_1\sqcup \ldots \sqcup H_\tau x^\tau_{m_\tau}$, where $m_\tau+1$ is the cardinality of the orbit space, and $x^\tau_0, x^\tau_1,\ldots,x^\tau_{m_\tau}\in \tau$ are the representatives.
One of these disjoint orbits
must be $H_\tau$, so we assume without loss of generality that 
$x^\tau_0=1$. Since the
action of $H_\tau$ on $\tau$ is free, all orbits have the same cardinality as $H_\tau$, 
from which
it follows that $n=|\tau|=(m_\tau+1)|H_\tau|$. 

Finally, let $y^\tau_1,\ldots,y^\tau_{m_\tau}\in G$ be as in (3) in the statement.
Then these elements determine a decomposition of $\tau$ as a disjoint union of 
$H_\tau$-orbits, so up to a permutation they must agree modulo $H_\tau$ with 
$x^\tau_1,\ldots,x^\tau_{m_\tau}$, as desired.
\end{proof}

 
\begin{nota}\label{nota:Section}
Let $G$ be a discrete group and let $n\in\N$. For $\tau\in\mathcal{T}_n(G)$,
we set 
$H_\tau=\{h\in G\colon h\tau=\tau\}$.
Using \autoref{lma:OrbitTuples}, we set $m_\tau=\frac{n}{|H_\tau|}-1$ and 
fix elements $x_0^\tau=1, x_1^\tau,\ldots,x_{m_\tau}^\tau\in G$ satisfying
\[\tau=H_\tau\sqcup H_\tau x^\tau_1\sqcup \ldots\sqcup H_\tau x^\tau_{m_\tau}.\]
We write $A_{G\cdot\tau}$ for the ideal $\sum_{g\in\tau^{-1}}A_{g\tau}$ of
$A$.
Whenever $\tau$ is understood from the context, we will omit it from the notation
for $H_\tau$, $m_\tau$ and $x_j^\tau$, for $j=1,\ldots,m_\tau$. For example, the 
above identity will be written $\tau=H\sqcup H x_1\sqcup \ldots\sqcup H x_{m}$. 

Let $\mathcal{O}_n(G)$ be the orbit space for the partial system described in \autoref{df:TnG}. We denote by $\kappa\colon \mathcal{T}_n(G)\to \mathcal{O}_n(G)$ the canonical
quotient map, and fix, for the rest of this work, a global section
$s\colon \mathcal{O}_n(G)\to \mathcal{T}_n(G)$ for it. 
For $z\in\mathcal{O}_n(G)$, we write $\tau_z$ for $s(z)$; we write 
$H_z$ for $H_{\tau_z}$; 
we write $m_z$ for $m_{\tau_z}$. (Note that $m_z$ is really independent of 
the choice of the section, unlike $H_z$ or $\tau_z$.)
\end{nota}

\begin{prop}\label{prop:lowerpa}
Let $n\in\N$, let $G$ be a discrete group, and let $\tau\in\mathcal{T}_n(G)$. Set $X=\{1,x_1,\ldots,x_m\}$ to be the set of representatives of $H_\tau$-classes.
For $g\in G$, set $X_g=\{x\in X\colon g\in x^{-1}\tau\}$. Then
\begin{enumerate}
\item For $g\in G$ and $x\in X_{g^{-1}}$ there is a
   unique $\sigma_g(x)\in X_g$ such that $g\in
   \sigma_g(x)^{-1}Hx$. 
\item $((X_g)_{g\in G}, (\sigma_g)_{g\in G})$ is
   a partial action of $G$ on $X$.   
\end{enumerate}
\end{prop}
\begin{proof}
Since $x_j^{-1}\tau=\bigsqcup_{k=0}^mx_j^{-1}Hx_k$, part~(1)  follows. As for~(2), note first that
$X_1=X$ and
$\sigma_1=\id_X$. Suppose that $x\in X_{g_2^{-1}}$ 
satisfies $\sigma_{g_2}(x)\in X_{g_1^{-1}}$, that is
\[g_2^{-1}\in
x^{-1}H\sigma_{g_2}(x) \ \mbox{ and } \ g_1^{-1}\in
\sigma_{g_2}(x)^{-1}H\sigma_{g_1}(\sigma_{g_2}(x)).\] Then 
$(g_1g_2)^{-1}\in x^{-1}H\sigma_{g_1}(\sigma_{g_2}(x))\subseteq
x^{-1}\tau$. This shows that $x\in X_{(g_1g_2)^{-1}}$, and
$\sigma_{g_1g_2}(x)=\sigma_{g_1}(\sigma_{g_2}(x))$ by
uniqueness of the left member of the equality. 
\end{proof}



\begin{prop}\label{prop:EquivDecomp}
Let $G$ be a discrete group, let $A$ be a \ca, let $n\in\N$, let 
$\alpha=((A_g)_{g\in G}, (\alpha_g)_{g\in G})$ be a partial action of $G$ on $A$ with the
$n$-decomposition property, and let 
$\tau\in\mathcal{T}_n(G)$. Adopt the conventions from \autoref{nota:Section}. Then:
\be
\item The restriction of $\alpha|_{H_\tau}$ to $A_\tau$ is a global action;
\item The ideal $A_{G\cdot \tau}$ is $G$-invariant, and for $g\in G$
  one has 
\[(A_{G\cdot \tau})_g=\begin{cases*}
      \ \ \ \ \ \ \ \ \ \ \ \{0\} & if $g\notin \tau^{-1}\cdot\tau$ \\
      \displaystyle{\sum\limits_{0\leq j\leq m\colon g\in x_j^{-1}\tau}}A_{x_j^{-1}\tau}& if $g\in \tau^{-1}\cdot\tau$.
    \end{cases*}\]
\item For $\sigma\in \mathcal{T}_n(G)$, we have 
$A_{G\cdot \sigma}\cap A_{G\cdot\tau}=
\{0\}$ if $\sigma\notin G\cdot \tau$.
\item There is a natural $G$-equivariant isomorphism
\[\varphi\colon \bigoplus\limits_{z\in\mathcal{O}_n(G)} A_{G\cdot \tau_z} \to A\]
given by $\varphi(a)=\sum_{z\in \mathcal{O}_n(G)} a_z$ for all
$a=(a_z)_{z\in\mathcal{O}_n(G)}$.
\ee
\end{prop}
\begin{proof}
(1). Since $H\tau=\tau$, it follows for every $h\in H$ that 
$A_h\cap A_\tau=A_\tau$ and moreover $\alpha_h(A_\tau)=A_\tau$. Hence $\alpha|_H$
induces a global action on $A_\tau$.

(2). Fix $g\in G$. We have $A_{g^{-1}}\cap A_\tau= \{0\}$
if $g^{-1}\notin \tau$; and $A_{g^{-1}}\cap A_\tau= A_\tau$ if $g^{-1}\in \tau$, in which case $\alpha_g(A_\tau)= A_{g\cdot \tau}$. Hence $A_{G\cdot \tau}$ is invariant and there is a well-defined restricted partial action of $G$ on it. For $g\in G$, we have 
\[(A_{G\cdot \tau})_g=A_{G\cdot \tau}\cap A_g=\sum\limits_{j=0}^{m}A_{x_j^{-1}\tau}\cap A_g=\sum\limits_{0\leq j\leq m\colon g\in x_j^{-1}\tau}A_{x_j^{-1}\tau},\]
where at the last equality we use the decomposition property.
In particular, if $g\notin \tau^{-1}\cdot\tau$, this domain is trivial. 

(3). We prove the contrapositive, so we
suppose that $A_{G\cdot \sigma}\cap A_{G\cdot\tau}\neq \{0\}$
and will show that $\sigma\in G\cdot \tau$. 
Fix $g\in \tau^{-1}$ and $h\in \sigma^{-1}$ with 
$A_{g\tau}\cap A_{h\sigma}\neq \{0\}$. 
By \autoref{rem:DisjTuples}, we have $g\tau=h\sigma$. In particular, 
$h\in g\tau$ and thus 
$(h^{-1}g)\tau=\sigma$, as desired.

(4). Note that $\varphi$ is a
homomorphism by part~(3) above. Moreover, 
it is clearly injective and equivariant, and it is surjective
by condition~(a) of \autoref{df:nIntProp}.
\end{proof}

\begin{rem}\label{rem:EquivDecomp}
Using part~(4) of the proposition above, a number of facts
about decomposable partial actions can be reduced to
the $G$-invariant direct summands $A_{G\cdot\tau_z}$ 
In practice, for many purposes it suffices
to work with a single tuple $\tau\in\mathcal{T}_n(G)$ and the induced partial
action on $A_{G\cdot \tau}$.
\end{rem}

The following lemma will be used repeatedly (see \autoref{nota:Section}).

\begin{lma}\label{rem:QuotientMapsAGtau}
Let $G$ be a discrete group, let $A$ be a \ca, let $n\in\N$, and let $\alpha$ be a partial 
action of $G$ on $A$ with the $n$-decomposition property.
Fix $\tau\in\mathcal{T}_n(G)$. 
\be\item There are canonical quotient 
maps $\pi_j\colon A_{G\cdot \tau}\to A_{x_j^{-1}\tau}$, for $j=0,\ldots,m$, which can be explicitly 
described as follows: if $(e_\lambda)_{\lambda\in\Lambda}$ is any 
approximate identity for $A_{x_j^{-1}\tau}$, then 
$\pi_j(a)=\lim_{\lambda}ae_\lambda$ for all $a\in A_{G\cdot \tau}$.
\item The map $\pi_j$ is independent of 
the choice of the representative in $Hx_j$.
\ee
\end{lma}
\begin{proof} (1). Follows from \autoref{rem:DisjTuples}, which implies that $A_{G\cdot\tau}\cong \bigoplus_{j=0}^{m}A_{x_j^{-1}\tau}$. 
	
(2). Let $y_0,\ldots,y_m\in G$ and $h_1,\ldots, h_m\in H$ be elements satisfying
$y_j=h_jx_j$ for all $j=0,\ldots,m$. The claim follows since $A_{y_j^{-1}\tau}=A_{x_j^{-1}\tau}$.
\end{proof}

In the context of the above lemma, and whenever the tuple $\tau$ is not 
clear from the context, we will write $\pi_j^\tau$ instead of $\pi_j$.

Our next goal is to show that the domains of partial actions with the
decomposition property admit
particularly well-behaved kind of approximate identities, which we call an ``equivariant system
of approximate identities''. 

\begin{df}\label{df:SystDynApproxId}
Let $G$ be a discrete group, and let 
$\alpha=((A_g)_{g\in G}, (\alpha_g)_{g\in G})$ be a partial action of $G$ on a \ca\ $A$. A \emph{system of equivariant approximate identities} for $\alpha$ is a choice, for every $g\in G$, of an approximate 
identity $(e_g^\lambda)_{\lambda\in\Lambda}$ of $A_g$, satisfying $\alpha_g(e_h^\lambda e_{g^{-1}}^\lambda)=e_{gh}^\lambda e_{g}^\lambda$
for all $g,h\in G$ and all $\lambda\in\Lambda$.
\end{df}

The above definition is inspired by the case of unital partial actions, where the units of the respective domains
form a system of equivariant approximate identities. Explicitly, if $1_g$ denotes the unit of $A_g$, then 
one has $\alpha_g(1_h1_{g^{-1}})=1_{gh}1_g$ for all $g,h\in G$. On the other 
hand, global actions of finite groups always admit such systems: it suffices to
consider a $G$-invariant approximate identity of $A$.

Systems of equivariant approximate identities fail to exist in general, but 
we show in the following proposition that decomposable systems
always possess them. 

\begin{prop}\label{thm:ApproxIdentitiesIntProp}
Let $G$ be a discrete group, let $A$ be a \ca, and let $\alpha=((A_g)_{g\in G}, (\alpha_g)_{g\in G})$ be a 
decomposable partial 
action of $G$ on $A$. Then there exists a system of equivariant approximate identities for $\alpha$.
\end{prop}
\begin{proof}
Let $n\in \N$ be such that $\alpha$ has the $n$-decomposition property. 
By \autoref{rem:EquivDecomp}, it suffices to show that $A_{G\cdot\tau}$
possesses a system of equivariant approximate identities for every $\tau\in\mathcal{T}_n(G)$. 
Fix $\tau\in\mathcal{T}_n(G)$ and fix an approximate identity 
$(e^{\lambda})_{\lambda\in \Lambda}$ for $A_{\tau}$. We use \autoref{nota:Section}. For $g\in G$ and $\lambda\in \Lambda$, define

\[e^{\lambda}_g=\frac{1}{|H|}\sum\limits_{j=0}^{m}\mathbbm{1}_{\{g\in x_j^{-1}\tau\}}\sum\limits_{h\in H}\alpha_{x_j^{-1}h}(e^{\lambda}). \]
Notice that $e^{\lambda}_g\in (A_{G\cdot\tau})_g$ is a positive contraction. We claim that $(e^{\lambda}_g)_{\lambda\in\Lambda}$ is an approximate
identity for $(A_{G\cdot\tau})_g$. Since $(A_{G\cdot\tau})_g=\{0\}$ whenever $g\notin \tau^{-1}\cdot \tau$, we may assume that $g\in \tau^{-1}\cdot \tau$. 
By part~(2) of \autoref{prop:EquivDecomp}, 
it is enough to show that $(e_g^{\lambda})$ is an approximate identity for $A_{x_j^{-1}\tau}$ for any $j=0,\ldots,m$ such that $g\in x_j^{-1}\tau$. Fix such $j$ and let $a\in A_{x_j^{-1}\tau}$.
Using at the first step that $\sum_{h\in H} \alpha_{x_k^{-1}h}(e^{\lambda})$ belongs to the ideal $A_{x_k^{-1}\tau}$ which is orthogonal to $A_{x_j^{-1}\tau}$ when $k\neq j$, and at the second 
that $(\alpha_{x_j^{-1}h}(e^{\lambda}))_{\lambda\in\Lambda}$ is an 
approximate identity of $A_{x_j^{-1}\tau}$, we conclude that
\[\lim_{\lambda\in\Lambda} e^{\lambda}_g a=
 \lim_{\lambda\in\Lambda}\frac{1}{|H|}\sum_{h\in H}\alpha_{x_j^{-1}h}(e^{\lambda})a=a.
\]

Let $g_1,g_2\in G$ and let $\lambda\in\Lambda$. We will show that 
$\alpha_{g_1}(e_{g_2}^\lambda e_{g_1^{-1}}^\lambda)=e_{g_1g_2}^\lambda e_{g_1}^\lambda$.
To ease notation, we drop the superscript $\lambda$ everywhere, as well as the denominator $|H|$. For $\ell=0,\ldots,m$, 
let $\pi_\ell\colon A_{G\cdot\tau}\to A_{x_\ell^{-1}\tau}$ be the map from \autoref{rem:QuotientMapsAGtau}. Then
\begin{align*}
\pi_\ell(e_{g_1g_2}e_{g_1})&=\sum\limits_{j,k=0}^{m}\sum\limits_{h,t\in H}\pi_\ell\big(\big(\mathbbm{1}_{\{g_1g_2\in x_j^{-1}\tau\}}\alpha_{x_j^{-1}h}(e)\big)\cdot \big(\mathbbm{1}_{\{g_1\in x_k^{-1}\tau\}}\alpha_{x_k^{-1}t}(e)\big)\big)\\
&=\sum\limits_{j=0}^{m}\sum\limits_{h,t\in H}\pi_\ell\big(\mathbbm{1}_{\{g_1g_2\in x_j^{-1}\tau\}}\mathbbm{1}_{\{g_1\in x_j^{-1}\tau\}}\alpha_{x_j^{-1}h}(e)\alpha_{x_j^{-1}t}(e)\big)\\
&=\sum\limits_{h,t\in H}\mathbbm{1}_{\{g_1g_2\in x_\ell^{-1}\tau\}}\mathbbm{1}_{\{g_1\in x_\ell^{-1}\tau\}} \alpha_{x_\ell^{-1}h}(e)\alpha_{x_\ell^{-1}t}(e).
\end{align*}

On the other hand,
\begin{align*}
\pi_\ell(\alpha_{g_1}(e_{g_2}e_{g_1^{-1}}))&=\sum\limits_{k=0}^{m}\sum\limits_{s,r\in H}\pi_\ell\big(\mathbbm{1}_{\{g_2\in x_k^{-1}\tau\}}\mathbbm{1}_{\{g_1^{-1}\in x_k^{-1}\tau\}}\alpha_{g_1}(\alpha_{x_k^{-1}s}(e)\alpha_{x_k^{-1}r}(e))\big)\\
&=\sum\limits_{k=0}^{m}\sum\limits_{s,r\in H}\mathbbm{1}_{\{g_2\in x_k^{-1}\tau\}}\mathbbm{1}_{\{g_1\in x_\ell^{-1}Hx_k\}}\alpha_{g_1}(\alpha_{x_k^{-1}s}(e)\alpha_{x_k^{-1}r}(e))\\
&=\sum\limits_{k=0}^{m}\sum\limits_{s,r\in H}\mathbbm{1}_{\{g_1g_2\in x_\ell^{-1}\tau\}}\mathbbm{1}_{\{g_1\in x_\ell^{-1}Hx_k\}}\alpha_{g_1}(\alpha_{x_k^{-1}s}(e)\alpha_{x_k^{-1}r}(e))\\
&=\sum\limits_{k=0}^{m}\sum\limits_{s,r\in H}\mathbbm{1}_{\{g_1g_2\in x_\ell^{-1}\tau\}}\mathbbm{1}_{\{g_1\in x_\ell^{-1}Hx_k\}}\alpha_{x_\ell^{-1}s}(e)\alpha_{x_\ell^{-1}r}(e)\\
&=\sum\limits_{s,r\in H}\mathbbm{1}_{\{g_1g_2\in x_\ell^{-1}\tau\}}\mathbbm{1}_{\{g_1\in x_\ell^{-1}\tau\}}\alpha_{x_\ell^{-1}s}(e)\alpha_{x_\ell^{-1}r}(e),
\end{align*}
where at the second and last steps we use that if $g\in x_k^{-1}\tau$, then there is a unique $j\in\{0,\ldots,m\}$ such that $g\in x_k^{-1}Hx_j$. The result follows since $\sum_{j=0}^{m} \pi_j=\id_{A_{G\cdot \tau}}$.
\end{proof}

\section{Enveloping actions}
In this section, we show that decomposable partial actions
are always globalizable, and we 
give an explicit construction of their globalizations. 
We point out that the results in this section apply without major
modifications to algebraic partial actions.


We begin by recalling the
notion of an enveloping action (Definition~3.4 in~\cite{Aba_enveloping_2003}).

\begin{df}\label{df:enveloping}
Let $G$ be a discrete group and let 
$\alpha=((A_g)_{g\in G}, (\alpha_g)_{g\in G})$ be a partial action of $G$ on a \ca\ 
$A$. A triple $(B,\beta,\iota)$ consisting of a \ca\ $B$, a (global) action
$\beta\colon G\to\Aut(B)$ and an embedding $\iota\colon A\to B$ as an ideal, is said to 
be an \emph{enveloping action} of $\alpha$ if the following conditions are
satisfied:
\be\item $A_g=A\cap \beta_g(A)$ for all $g\in G$;
\item $\alpha_g(a)=\beta_g(a)$ for all $a\in A_{g^{-1}}$ and all $g\in G$;
\item $B=\overline{\mathrm{span}}\{\beta_g(a)\colon a\in A, \ g\in G \}$.\ee

We say that $\alpha$ is \emph{globalizable} if there exists an enveloping action.
\end{df}

By Theorem~3.8 in~\cite{Aba_enveloping_2003}, enveloping actions are unique up to
an equivariant isomorphism extending the identity on $A$; for this reason, we
will always refer to \emph{the} enveloping action of a given globalizable 
partial action.

Not every partial action is globalizable, and even when it is, identifying its
enveloping action may turn out to be challenging. Using Ferraro's recent abstract
characterization of globalizability \cite{Fer_construction_2018}, we show next that decomposable partial
actions are always globalizable. Obtaining an 
explicit description of its enveloping action is rather involved, and this is 
done in \autoref{thm:InterPropGlobz}.

\begin{prop}
Let $G$ be a discrete group, let $A$ be a \ca, and let 
$\alpha$ be a decomposable partial action of $G$ on 
$A$. Then $\alpha$ is globalizable. 
\end{prop}
\begin{proof}
By the equivalence between (a) and (c) in~Theorem~4.5 of~\cite{Fer_construction_2018}, it suffices to show that 
given $(g,a,b)\in G\times A\times A$, there exists
$u_{g,a,b}\in A_g$ such that
$cu_{g,a,b}=\alpha_g(\alpha_{g^{-1}}(c)a)b$ for all $c\in A_g$.
Let $n\in\N$ such that $\alpha$ has the $n$-decomposition property. 
By condition~(a) in \autoref{df:nIntProp}, we may take
$a\in A_\tau$ for $\tau\in\mathcal{T}_n(G)$. 

For $g\notin\tau^{-1}$, 
we set $u_{g,a,b}=0$. In this case,
by condition~(b) in \autoref{df:nIntProp} we have $A_\tau\cap A_{g^{-1}}=\{0\}$, and hence the identity $cu_{g,a,b}=\alpha_g(\alpha_{g}^{-1}(c)a)b$ holds for all $c\in A_g$, since $\alpha_{g}^{-1}(c)a=0$.
For $g\in\tau^{-1}$, set $u_{g,a,b}=\alpha_g(a)b$, which
is well-defined as $a\in A_{g^{-1}}$. Then
\[cu_{g,a,b}=c\alpha_g(a)b=\alpha_g(\alpha_{g^{-1}}(c)a)b\]
for all $c\in A_{g}$, as desired. This finishes the proof.
\end{proof}

Our next goal is to obtain an explicit description of the globalization of a decomposable
partial action, which
was shown to exist in the proposition above.
We need to recall the notion of induced (global) actions; see Chapter~3 
in~\cite{Wil_crossed_2007}. 

\begin{df} \label{df:Ind}
Given a discrete group $G$, a subgroup $H$, a \ca\ $C$ and an action
$\gamma\colon H\to \Aut(C)$, the \emph{induced dynamical system}
$(\mathrm{Ind}_H^G(C),\mathrm{Ind}_H^G(\gamma))$ is
\[\mathrm{Ind}_H^G(C)=
\left\{\xi\in C_b(G,C)\colon 
\begin{tabular}{@{}l@{}}
$\xi(hg)=\gamma_h(\xi(g))$ for all $g\in G$ and $h\in H$,\\ 
and the map $Hg\mapsto \|\xi(g)\|$ is in $C_0(G/H)$
   \end{tabular}
\right\},\]
with $\mathrm{Ind}_H^G(\gamma)_g(\xi)(k)=\xi(kg)$ for all $g,k\in G$ and all $\xi\in
\mathrm{Ind}_H^G(C)$.
 \end{df}

\begin{rem}\label{rem:green}
For discrete groups, Green's Imprimitivity Theorem
(see, for example, Corollary~4.17
in~\cite{Wil_crossed_2007})
can be deduced from the theory of enveloping actions of partial actions from \cite{Aba_enveloping_2003,AbaMar_amenability_2009}. Indeed, 
in the context of \autoref{df:Ind}, the action $\gamma$
defines a partial
action $\widetilde{\gamma}$ of $G$ on $C$, where 
$\widetilde{\gamma}_g=\gamma_g$ if $g\in H$, and
$\widetilde{\gamma}_g=0$ otherwise. 
Then $C\rtimes_\gamma 
H=C\rtimes_{\widetilde{\gamma}}G$, and similarly for the reduced
crossed products.
Moreover, one readily checks
that $(\mathrm{Ind}_H^G(C),\mathrm{Ind}_H^G(\gamma))$ is an enveloping
action for $\widetilde{\gamma}$, with inclusion map 
$\iota\colon
C\to\mathrm{Ind}_H^G(C)$ given by 
$\iota(c)(g)= \gamma_g(c)$ if $g\in H$, and $0$ if $g\notin H$, for 
all $c\in C$.
It follows from \cite{Aba_enveloping_2003} and
\cite{AbaMar_amenability_2009} that the crossed products
$C\rtimes_\gamma 
H=C\rtimes_{\widetilde{\gamma}}G$ and
$\mathrm{Ind}_H^G(C)\rtimes_{\mathrm{Ind}_H^G(\gamma)}G$ are Morita
equivalent, as well as the reduced crossed products     
$C\rtimes_{\gamma,r}H=C\rtimes_{\widetilde{\gamma},r}G$ and
$\mathrm{Ind}_H^G(C)\rtimes_{\mathrm{Ind}_H^G(\gamma),r}G$.       
\end{rem}

\begin{rem}\label{rem:Amenability}
We recall \cite{AbaMar_amenability_2009} 
that a partial action is said to be \emph{amenable}
if the canonical surjection $A\rtimes_\alpha G\to
A\rtimes_{\alpha,r} G$ is injective. 
Adopt the notation and assumptions of \autoref{rem:green}. 
By Corollary~1.3 of \cite{AbaMar_amenability_2009} (or
Corollary 5.4 of \cite{AbaFerr_applications_2017}),
the actions $\gamma$, $\widetilde{\gamma}$ and
$\mathrm{Ind}_H^G(\gamma)$ are simultaneously 
amenable or non-amenable. If $H$ is amenable, then they 
are all amenable.
\end{rem}

We are now ready for the main result of this section.

\begin{thm}\label{thm:InterPropGlobz}
Let $G$ be a discrete group, let $A$ be a \ca, let $n\in\N$, and let 
$\alpha=((A_g)_{g\in G}, (\alpha_g)_{g\in G})$ be a partial action of $G$ on $A$ with the
$n$-decomposition property. 
For every $\tau\in \mathcal{T}_n(G)$, the restriction of $\alpha$ to $A_{G\cdot \tau}$ is
globalizable and amenable, and its enveloping action is
$(\mathrm{Ind}_{H_\tau}^G(A_{\tau}),\mathrm{Ind}_{H_\tau}^G(\alpha),\iota_\tau)$,
where the inclusion
$\iota_\tau\colon A_{G\cdot \tau} \to \mathrm{Ind}_{H_\tau}^G(A_{\tau})$
is given by 
\[\iota_\tau(a)(g)=\begin{cases*}
      \alpha_{g}(\pi^\tau_k(a)) & if $g\in H_\tau x_k$ for some $k=0,\ldots,m_\tau$ \\
      0 & else.
    \end{cases*}
\]
for all $a\in A_{G\cdot \tau}$ and all $g\in G$. (In particular, $\iota_\tau(a)$
is supported on $\tau$.)
It follows that $\alpha$ is globalizable and amenable, and its globalization is
\[\bigoplus\limits_{z\in \mathcal{O}_n(G)}(\mathrm{Ind}_{H_z}^G(A_{\tau_z}),\mathrm{Ind}_{H_z}^G(\alpha),\iota_{\tau_z}).\]
\end{thm}
\begin{proof}
As explained in \autoref{rem:EquivDecomp}, it suffices to prove the first 
part of the statement.
Let $\tau\in \mathcal{T}_n(G)$. To lighten the notation, we abbreviate $H_\tau$
to $H$, $m_\tau$ to $m$, $x_j^\tau$ to $x_j$, and $\iota_\tau$ to $\iota$. 
It is clear that $\iota$ is an embedding. We check the remaining conditions in
\autoref{df:enveloping} in a series of claims.

\textbf{Claim:} \emph{the range of $\iota$ is contained in 
$\mathrm{Ind}_H^G(A_\tau)$.} 
It is clear that $\iota_\tau(a)(g)$ belongs to $A_\tau$ for all $a\in A_{G\cdot\tau}$ and all $g\in G$.
Let $g\in G$, let $h\in H$, and let $a\in A_{G\cdot\tau}$. Then $\iota(a)(hg)=0=\iota(a)(g)$ unless $g\in Hx_k$ for some $k=0,\ldots,m$, in which case
\[\iota(a)(hg)= \alpha_{hg}(\pi^\tau_k(a))=\alpha_h(\iota(a)(g)),\]
as desired. Moreover, the induced map $Hg\mapsto \|\iota(a)(g)\|$ belongs
to $C_0(G/H)$ because it is supported on the finite set $\{Hx_0,\ldots,Hx_m\}$. 
This proves the claim.

\textbf{Claim:} \emph{the image of $\iota$ is an ideal in $\mathrm{Ind}_H^G(A_\tau)$.}
Let $\xi\in \mathrm{Ind}_H^G(A_\tau)$, let $k=0,\ldots, m$, and 
let $a\in A_{x_k^{-1}\tau}$. 
It suffices to check that $\iota(a)\xi$ belongs to the image of $\iota$. 
For $g\in G$, we have 
$(\iota(a)\xi)(g)=0$ unless $g\in Hx_k$, in which case we have
\begin{align*}
(\iota(a)\xi)(g)=\alpha_g(a\alpha_{g^{-1}}(\xi(g)))
 =\alpha_g(a\alpha_{x_k}^{-1}(\xi(x_k)))
=\iota(a\alpha_{x_k}^{-1}(\xi(x_k)))(g),
\end{align*}
where at the last step we use that $a\alpha_{x_k}^{-1}(\xi(x_k))$
belongs to  $A_{x_k^{-1}\tau}$. We conclude that, $\iota(a)\xi=\iota(a\alpha_{x_k}^{-1}(\xi(x_k)))$, thus proving the claim. 
From now on, we set $\beta=\mathrm{Ind}_H^G(\alpha)$.

\textbf{Claim:} 
\emph{for $g\in G$, we have
\[\iota(A_{G\cdot \tau}\cap A_g)=\beta_g(\iota(A_{G\cdot\tau}))\cap \iota(A_{G\cdot \tau}).\ \]}
We begin with the inclusion $\subseteq$. For this, it is enough to show that 
$\iota(a)$ belongs to $\beta_g(\iota(A_{G\cdot\tau}))$
for $k=0,\ldots,m$ and $a\in A_{x_k^{-1}\tau}\cap A_g$. This is trivial if 
$g\notin x_k^{-1}\tau$, since the decomposition property implies that 
$A_{x_k^{-1}\tau}\cap A_g=\{0\}$. Assume that $g\in x_k^{-1}\tau$ and let $j=0,\ldots, m$ satisfy $g\in x_k^{-1}Hx_j$.
For $t\in G$ we have 
\begin{align*}
\beta_g(\iota(\alpha_g^{-1}(a)))(t)&=\iota(\alpha_g^{-1}(a))(tg)
=\mathbbm{1}_{\{tg\in Hx_j\}}\alpha_{tg}(\alpha_g^{-1}(a))
=\mathbbm{1}_{\{t\in Hx_k\}}\alpha_t(a)=\iota(a)(t).
\end{align*}
It follows that $\iota(a)=\beta_g(\iota(\alpha_g^{-1}(a)))$.
To prove the converse inclusion, it suffices 
to show
\[\iota(A_{G\cdot \tau}\cap A_g)\supseteq \beta_g(\iota(A_{x_j^{-1}\tau}))\cap\iota(A_{x_k^{-1}\tau})\]
for all $g\in G$ and all $j,k=0,\ldots,m$. 
Let $a\in A_{x_j^{-1}\tau}$ and $b\in A_{x_k^{-1}\tau}$ satisfy 
$\beta_g(\iota(a))=\iota(b)$. We will be done once we show that $b\in A_g$. For $t\in G$ we have
\[\beta_g(\iota(a))(t)=\iota(a)(tg)=\begin{cases*}
      \alpha_{tg}(a) & if $tg\in Hx_j$\\
      0 & else.
    \end{cases*}\]
On the other hand
$\iota(b)(t)=\alpha_{t}(b)$ if $t\in Hx_k$, and 0 otherwise.
If $\beta_g(\iota(a))=\iota(b)=0$, then by injectivity of $\iota$ we get $b=0\in A_g$, as required. Otherwise, there exists some $t\in Hx_k$ such that $tg\in Hx_j$. Then, $g\in x_k^{-1}Hx_j\subseteq x_k^{-1}\tau$, and by definition 
$A_{x_k^{-1}\tau}\subseteq A_g$, so $b\in A_g$, as required. The claim is proved.

\textbf{Claim:} \emph{$\iota(\alpha_g(a))=\beta_g(\iota(a))$ for all $a\in (A_{G\cdot\tau})_{g^{-1}}$ and all $g\in G$.} As $a=\sum_{j=0}^{m}\pi_j(a)$, we may assume that $a\in A_{x_j^{-1}\tau}\cap A_{g^{-1}}$ for some $j=0,\ldots, m$, and that $g^{-1}\in x_j^{-1}\tau$ (otherwise, this intersection is trivial). Let $k\in\{0,\ldots,m\}$ be the unique element such that $g^{-1}\in x_j^{-1}Hx_k$. Note that $a=\pi_j(a)$.
For $t\in G$, we have 
\[\beta_{g}(\iota(a))(t)=0=\iota(\alpha_{g}(a))(t)\]
unless $t$ belongs to $Hx_k$, in which case we have
\begin{align*}
\beta_{g}(\iota(a))(t)=\iota(a)(tg)
=\alpha_{tg}(a)
=\alpha_t(\alpha_{g}(a))
=\iota(\alpha_{g}(a))(t).
\end{align*}
This proves the claim.

\textbf{Claim:} \emph{$\overline{\sum_{g\in G}\beta_g(\iota(A_{G\cdot\tau}))}=\mathrm{Ind}_H^G(A_\tau)$.} Let $(g_\lambda)_{\lambda\in\Lambda}$ be a subset 
of $G$ satisfying 
$\bigsqcup_{\lambda\in\Lambda} Hg_\lambda=G$. 
An element $\xi\in \mathrm{Ind}_H^G(A_\tau)\subseteq C_b(G,A_\tau)$ 
is completely determined by its values on $g_\lambda$, for $\lambda\in \Lambda$. 
Moreover, the set of those $\xi\in \mathrm{Ind}_H^G(A_\tau)$ that take nonzero
values on a finite subset of $\{g_\lambda\in\Lambda\}$ is dense in 
$\mathrm{Ind}_H^G(A_\tau)$. 

Let $\xi\in \mathrm{Ind}_H^G(A_\tau)$ satisfy $\xi(g_\lambda)=0$ for all but finitely
many $\lambda$. 
By the comments above, it suffices to show that
$\sum\limits_{\lambda\in\Lambda}\beta_{g_\lambda}^{-1}(\iota(\xi(g_\lambda)))=\xi$.
To see this, let $\mu\in \Lambda$
and $h\in H$. Then
\begin{align*}
\Big(\sum\limits_{\lambda\in\Lambda}\beta_{g^{-1}_\lambda}(\iota(\xi(g_\lambda)))\Big)(hg_\mu)=\sum\limits_{\lambda\in\Lambda}\iota(\xi(g_\lambda))(hg_\mu g_\lambda^{-1})=\xi(hg_\mu),
\end{align*}
where at the second step we use the fact that $hg_\mu g_\lambda^{-1}\in H$ if and only if $\mu=\lambda$. This proves the claim, and finishes the identification of the enveloping action. 
Amenability of the 
restriction of $\alpha$ to $A_{G\cdot\tau}$ (and hence for 
$\alpha$) follows from \autoref{rem:Amenability}, since 
the global action $\alpha\colon H\to\Aut(A_\tau)$ is amenable
because $H$ is finite.
\end{proof}

In the next result, 
we write $\sim_M$ for Morita equivalence.

\begin{cor}\label{cor:CPuptoMorEquiv}
Fix $\tau\in \mathcal{T}_n(G)$. With the assumptions of
\autoref{thm:InterPropGlobz} we have 
\[A_{G\cdot \tau}\rtimes_\alpha G=A_{G\cdot \tau}\rtimes_{\alpha,r} G\sim_M \mathrm{Ind}_{H_\tau}^G(A_\tau)\rtimes_{\mathrm{Ind}_{H_\tau}^G(\alpha)} G\sim_M
A_{\tau}\rtimes_{\alpha|_{H_\tau}} H_\tau.\] 
Moreover, 
$A\rtimes_\alpha G= A\rtimes_{r,\alpha} G\sim_M \bigoplus_{z\in\mathcal{O}_n(G)}
A_{\tau_z}\rtimes_{\alpha|_{H_z}} H_z$. In particular, the partial crossed
products of $\alpha$ can be computed, up to Morita equivalence,
using the crossed products of the global systems
$H_\tau\curvearrowright A_{\tau}$.
\end{cor}

\section{Fixed point algebras}
In this section, we study algebras of $G$-invariant elements. For global actions, the theory is well-known to work best for 
finite (or even compact) groups: there are a faithful conditional expectation $E\colon A\to A^G$, 
and an injective homomorphism $c\colon A^G\to A\rtimes G$ whose
image is a corner in $A\rtimes G$. With $u_g\in M(A\rtimes G)$, for $g\in G$, denoting the canonical unitaries, these maps are given by
\[E(a)=\frac{1}{|G|}\sum_{g\in G}\alpha_g(a) \ \ \mbox{ and } \ \ c(b)=\frac{1}{|G|}\sum_{g\in G}u_gb\]
for all $a\in A$ and all $b\in A^G$.
Both of these maps fail to exist for infinite groups: 
in this case, $A^G$ is typically too small (and is often trivial). 
For partial actions, even of finite groups, additional 
complications arise; see \autoref{eg:NoCondExp} and \autoref{eg:CornerMapNotExist}.

In this section, we show that decomposable partial actions, even for infinite groups, behave similarly to \emph{global} actions of 
\emph{finite} groups from the point of view of fixed point algebras. 
For example, there
are analogs of the maps $E\colon A\to A^G$ and $c\colon A^G\to A\rtimes G$
mentioned above; see \autoref{prop:CondExp} and \autoref{prop:CornerMap}. We also give
an explicit description of $A^G$ in terms of the fixed point algebras of the 
global systems $H_\tau\curvearrowright A_\tau$;
see \autoref{prop:FixedPtAlg}.

\begin{df}\label{df:GInvariantElement}
Let $G$ be a discrete group, and let 
$\alpha=((A_g)_{g\in G}, (\alpha_g)_{g\in G})$ be a partial action of $G$ on a \ca\ $A$. We say that an element
$a\in A$ is \emph{$G$-invariant}, if for every $g\in G$ and every $b_{g}\in A_{g}$ we have
$\alpha_{g^{-1}}(ab_g)=a\alpha_{g^{-1}}(b_g)$. Finally, the subalgebra $A^G$
of $A$ consisting of all $G$-invariant elements is called the 
\emph{fixed point algebra} of $\alpha$.
\end{df}

In the context of the above definition, we could have instead
defined an element $a\in A$ to be $G$-invariant if 
$\alpha_{g^{-1}}(b_ga)=\alpha_{g^{-1}}(b_g)a$ for all $g\in G$
and all $b_g\in A_g$. The existence of approximate identities in 
\ca s implies that the two notions would be equivalent, as 
we show in the next remark.

\begin{rem}\label{lma:FixedCstar}
Let $G$ be a discrete group, and let 
$\alpha=((A_g)_{g\in G}, (\alpha_g)_{g\in G})$ be a partial action of $G$ on a \ca\ $A$. Let $a\in A^G$, let $g\in G$ and let 
$b_g\in A_g$; we claim that $\alpha_{g^{-1}}(b_ga)=\alpha_{g^{-1}}(b_g)a$. To see this, let $(e_j)_{j\in I}$ be an approximate identity of $A_{g}$. Then
\begin{align*}
\alpha_{g^{-1}}(b_ga)&=\lim_{j\in I} \alpha_{g^{-1}}(b_ga)\alpha_{g^{-1}}(e_j)
=\lim_{j\in I} \alpha_{g^{-1}}(b_gae_j)\\
&=\lim_{j\in I} \alpha_{g^{-1}}(b_g)\alpha_{g^{-1}}(ae_j)
=\lim_{j\in I} \alpha_{g^{-1}}(b_g)a\alpha_{g^{-1}}(e_j)
=\alpha_{g^{-1}}(b_g)a,
\end{align*}
as required. In particular, it follows that $A^G$ is closed
under the adjoint operation, and is thus a $C^*$-subalgebra
of $A$.
\end{rem}

Perhaps surprisingly, conditional expectations as described at
the beginning of this section
do not always exist in the case of partial actions of finite groups: 

\begin{eg}\label{eg:NoCondExp}
Set $X=(0,2]$ and $U=(0,1)\cup (1,2)\subseteq X$, and let $\sigma\colon U\to U$ be the order-2 homeomorphism given by $\sigma(x)=\begin{cases*}
      x+1 & if $x<1$ \\
      x-1 & if $x>1$.
    \end{cases*}$ for $x\in U$. Then there is no conditional
expectation $E\colon C_0(X)\to C_0(X)^{\Z_2}$.
\end{eg}
\begin{proof}
We claim that
\[C_0(X)^{\Z_2}=\{f\in C_0((0,2])\colon f(x)=f(x+1) \mbox{ for all } 0<x<1\}.\]
Since the inclusion $\supseteq$ is clear, 
we prove the reverse containment. Let $f\in C_0(X)^{\Z_2}$, and let 
$g\in C_0(U)$ be strictly positive. For $x\in (0,1)\subseteq U$ we have 
\[f(x+1)g(x+1)=\alpha_1(fg)(x)=\left(f\alpha_1(g)\right)(x)=f(x)g(x+1).\]
It follows that $f(x)=f(x+1)$ for all $x\in (0,1)$, as desired. 

In particular, any $f\in C_0(X)^{\Z_2}$ satisfies $f(1)=0$.
Suppose that there exists a conditional expectation 
$C_0(X)\to C_0(X)^{\Z_2}$. Denote by $\iota\in C_0(X)$ the canonical inclusion
$X\hookrightarrow \C$. We claim that $E(\iota)(1)\neq 0$, which will be
a contradiction.

For every $n\in\N$, denote by $f_n\in C_0(X)^{\Z_2}$ any positive function with 
$f_n\leq \iota$ that
satisfies $f_n(x)=x$ for $x\in (0,\frac{n-1}{n}]$. 
Since conditional expectations are order-preserving, we must have 
$f_n=E(f_n)\leq E(\iota)$ for all $n\in\N$. 
In particular, 
\[E(\iota)(1)=\lim_{n\to \I} E(\iota)\left(\frac{n-1}{n}\right)\geq 
 \lim_{n\to \I} f_n\left(\frac{n-1}{n}\right)=\lim_{n\to\I}\frac{n-1}{n}=1,
\]
as desired. It follows that no such conditional expectation exists. \end{proof}


In contrast with the previous example, we show in 
\autoref{prop:CondExp} that decomposable actions,
even of infinite 
groups, always admit canonical conditional expectations onto their fixed point
algebras. We retain
\autoref{nota:Section}.  

\begin{thm}\label{prop:CondExp}
Let $G$ be a discrete group, let $A$ be a \ca, let $n\in\N$, and let $\alpha$ be a partial 
action of $G$ on $A$ with the $n$-decomposition property.
Fix $\tau\in\mathcal{T}_n(G)$.  
\be
\item For $h\in H$, 
$j,k=0,\ldots,m$, and $a\in A_{G\cdot\tau}^G$, we have
$\pi_k(a)=\alpha_{{x_k^{-1}}h{x_j}}(\pi_j(a))$.
\item For $a\in A_{\tau}$, the following element belongs to $A_{G\cdot\tau}^G$:
\[\varphi_\tau(a)=\frac{1}{|H|}\sum\limits_{j=0}^m  \sum\limits_{h\in H} \alpha_{hx_j}^{-1}(a).\]
\item The map $E_\tau\colon A_{G\cdot \tau}\to A_{G\cdot\tau}^G$ given by
\[E_\tau(a)=\frac{1}{m+1}\sum\limits_{j=0}^m \varphi_\tau\left(\alpha_{x_j}(\pi_j(a))\right),\]
for all $a\in A_{G\cdot \tau}$, is a faithful conditional expectation.
\item $E_\tau$ is independent of the elements $x_1,\ldots,x_m$.
\item For $g\in \tau^{-1}$, we have $E_{g\tau}=E_\tau$.
\ee

It follows that the map $E=\bigoplus\limits_{z\in \mathcal{O}_n(G)}E_{\tau_z}\colon A\to A^G$ is a canonical faithful conditional expectation, that is, $E$
is independent of the section $z\mapsto \tau_z$.
\end{thm}
\begin{proof}
(1). Let $(e_{\lambda})_{\lambda\in\Lambda}$ be an approximate identity for $A_{x_k^{-1}\tau}$. Then $\big(\alpha_{x_k^{-1}h{x_j}}^{-1}(e_\lambda)\big)_{\lambda\in\Lambda}$ is an approximate identity for $A_{x_j^{-1}\tau}$, since $\alpha_{x_k^{-1}h{x_j}}^{-1}(A_{x_k^{-1}\tau})=A_{x_j^{-1}\tau}$. Using this in combination with 
part~(1) of \autoref{rem:QuotientMapsAGtau} at 
the first and third step, and using invariance of $a$ at the second step, we
get
\begin{align*}
\alpha_{x_k^{-1}h{x_j}}^{-1}(\pi_k(a))
=\lim_\lambda \alpha_{x_k^{-1}h{x_j}}^{-1}(ae_{\lambda})
=\lim_{\lambda} a\alpha_{x_k^{-1}h{x_j}}^{-1}(e_{\lambda})
=\pi_j(a).
\end{align*}

(2). 
Let $\tau\in\mathcal{T}_n(G)$, let $a\in A_{\tau}$, let $g\in G$, and 
let $b_g\in (A_{G\cdot \tau})_g$. We claim that
\[\alpha_{g^{-1}}(\varphi_\tau(a)b_g)=\varphi_\tau(a)\alpha_{g^{-1}}(b_g).\]
Note first that the identity is trivially satisfied if $g\notin \tau^{-1}\tau$, 
since in this case $b_g$ must be zero. Since $b_g=\sum_{j=0}^{m} \pi_j(b)$, we may assume that $b_g\in A_{x_j^{-1}\tau}\cap A_g$ for some $j=0,\ldots, m$ and that $g\in x_j^{-1}\tau$ (otherwise, this intersection is zero).
Let $k\in \{0,\ldots,m\}$ and $h\in H$ be the 
unique elements satisfying $g=x_j^{-1}hx_k$. Then
\[\alpha_{g^{-1}}(\varphi_\tau(a)b_g)= \frac{1}{|H|}\sum\limits_{\ell=0}^m  \sum\limits_{t\in H} \alpha_{g^{-1}}(\alpha_{tx_\ell}^{-1}(a)b_g).\]
The product $\alpha_{tx_\ell}^{-1}(a)b_g$ belongs to $A_{x_\ell^{-1}\tau}\cap 
A_{x_j^{-1}\tau}$, so it is zero unless $\ell=j$. Similarly, 
\[\varphi_\tau(a)\alpha_{g^{-1}}(b_g)=\frac{1}{|H|}\sum\limits_{\ell=0}^m  \sum\limits_{t\in H} \alpha_{tx_\ell}^{-1}(a)\alpha_{g^{-1}}(b_g)\]
and the product $\alpha_{tx_\ell}^{-1}(a)\alpha_{g^{-1}}(b_g)$ is zero unless
$\ell=k$. Thus, it suffices to show that
\[\sum_{t\in H}\alpha_{x_k^{-1}h^{-1}x_j}(\alpha_{tx_j}^{-1}(a)b_g)=
 \sum_{t\in H}\alpha_{tx_k}^{-1}(a)\alpha_{x_k^{-1}h^{-1}x_j}(b_g).
\]
Fix $t\in H$. Then $\alpha_{tx_j}^{-1}(a)$ belongs to the domain of 
$\alpha_{x_k^{-1}h^{-1}x_j}$, and thus
\begin{align*}\alpha_{x_k^{-1}h^{-1}x_j}(\alpha_{tx_j}^{-1}(a)b_g)&=
 \alpha_{x_k^{-1}h^{-1}x_j}(\alpha_{tx_j}^{-1}(a))\alpha_{x_k^{-1}h^{-1}x_j}(b_g)
 =
 \alpha_{thx_k}^{-1}(a)\alpha_{x_k^{-1}h^{-1}x_j}(b_g),
\end{align*}
which implies the identity above. Hence $\varphi_\tau(a)$ belongs to $A_{G\cdot\tau}^G$.

(3). Note that $E_\tau$ is well-defined, since
$\alpha_{x_j}(\pi_j(a))$ belongs to $A_\tau$ for all $j=0,\ldots,m$ and 
all $a\in A_{G\cdot\tau}$. Moreover, for $a\in A_{G\cdot\tau}$, we have
\[E_\tau(a)=\frac{1}{|H|(m+1)}\sum_{j,k=0}^m \sum_{h\in H} \alpha_{x_j^{-1}hx_k}(\pi_k(a)).\]
We claim that $E_\tau$ is faithful. Observe that, for $a\geq 0$, we have 
$E_\tau(a)=0$ if and only if $\pi_k(a)=0$ for all $k=0,\ldots,m$.
The claim follows since $a=\sum_{k=0}^m\pi_k(a)$. 

Fix $a\in A_{G\cdot\tau}^G$. Then 
$\alpha_{x_j^{-1}hx_k}(\pi_k(a))=\pi_j(a)$ for all $j,k=0,\ldots,m$
and all $h\in H$, by part~(1) of this theorem. Thus,
\[E_\tau(a)=\frac{1}{|H|(m+1)}\sum_{j,k=0}^m \sum_{h\in H} \alpha_{x_j^{-1}hx_k}(\pi_k(a)) = \frac{1}{|H|(m+1)}\sum_{j,k=0}^m \sum_{h\in H} \pi_j(a)=a,\]
as desired. It follows that $E_\tau$ is an idempotent map. Since it is
easily seen to be contractive, it follows that $E_\tau$ is a conditional
expectation.

(4). We write $E_\tau^x$ for the conditional expectation constructed from the 
elements $x_0=1,x_1,\ldots,x_m$ as in part~(2). 
Let $y_0=1,y_1,\ldots,y_m\in G$ be elements satisfying 
$\tau=\bigsqcup_{j=0}^m Hy_j$, and write $E_\tau^y$ for the corresponding 
conditional expectation. We want to show that $E_\tau^x=E_\tau^y$. 
By part~(3) of \autoref{lma:OrbitTuples}, we may assume without loss of 
generality that there exist $h_1,\ldots,h_m\in H$ satisfying
$y_j=h_jx_j$ for all $j=1,\ldots,m$. 
Let $a\in A_{G\cdot \tau}$. Then 
\begin{align*}
E^y_\tau (a)&= 
\frac{1}{|H|(m+1)}\sum\limits_{j,k=0}^m  \sum\limits_{h\in H} \alpha_{hh_jx_j}^{-1} 
\left(\alpha_{h_kx_k}(\pi_k(a))\right)\\ 
&=
\frac{1}{|H|(m+1)}\sum\limits_{j,k=0}^m  \sum\limits_{h\in H} \alpha_{x_j}^{-1}(\alpha_{h_j^{-1}hh_k}(\alpha_{x_k}(\pi_k(a))))=E_\tau^x(a).
\end{align*}

(5). Let $g\in \tau^{-1}$, and let $\ell=0,\ldots,m$ be the 
unique element satisfying $g\in x_\ell^{-1}H$. Then $g\tau=x_\ell^{-1}\tau$. In
particular, it suffices to assume that $g=x_\ell^{-1}$. 
Set $\sigma=x_\ell^{-1}\tau$.
By part~(4) of this theorem, we may compute $E_{\sigma}$ using any 
decomposition 
$\sigma=H_\sigma\sqcup H_\sigma y_1\sqcup \ldots \sqcup H_\sigma y_m$ 
as in \autoref{lma:OrbitTuples}. In this context, we must have 
$H_\sigma=x_\ell^{-1}H_\tau x_\ell$, and we take $y_j=x_\ell^{-1}x_j$ for 
$j=0,\ldots,m$. Note that $\pi_j^\tau=\pi_j^\sigma$ for all $j=0,\ldots,m$.
Given $a\in A_{G\cdot\tau}$, we have
\begin{align*}
E_{\sigma}(a)&= \frac{1}{|H_\sigma|(m+1)}\sum_{j,k=0}^m\sum_{t\in H_\sigma}
\alpha_{y_j^{-1}ty_k}(\pi_k^\sigma(a))\\
&= \frac{1}{|H_\tau|(m+1)}\sum_{j,k=0}^m\sum_{h\in H_\tau}
\alpha_{x_j^{-1}x_\ell(x_\ell^{-1} hx_\ell) x_\ell^{-1}x_k}(\pi_k^\tau(a))
=E_{\tau}(a).
\end{align*}
 The last statement of the theorem follows from part~(4) of
\autoref{prop:EquivDecomp}.
\end{proof}

Of particular importance is the existence of approximate identities consisting
of $G$-invariant elements. In the setting of global actions, a straightforward 
averaging argument shows that finite group actions
admit invariant approximate identities. 
On the other hand, invariant approximate identities fail to exist in general
for partial actions of finite groups; see \autoref{eg:NoCondExp}.
For decomposable actions, even of infinite groups, 
we establish the existence of $G$-invariant
approximate identities in the following proposition.

\begin{prop}
Let $G$ be a discrete group, let $A$ be a \ca, and let $\alpha=((A_g)_{g\in G}, (\alpha_g)_{g\in G})$ be a
decomposable partial 
action of $G$ on $A$. Then
there exists a $G$-invariant approximate identity for $A$.
\end{prop}
\begin{proof}
Let $E\colon A\to A^G$ be the canonical conditional expectation
constructed in \autoref{prop:CondExp}. If $(a_\lambda)_{\lambda\in\Lambda}$ 
is an approximate identity for $A$, then $(E(a_\lambda))_{\lambda\in\Lambda}$ 
is an approximate identity for $A^G$. We claim that it is also an 
approximate identity for $A$. Let $n\in\N$ be such that $\alpha$ has the 
$n$-decomposition property.
Using part~(4) of \autoref{prop:EquivDecomp}, 
it suffices to show that $(E_\tau(a_\lambda))_{\lambda\in\Lambda}$
is an approximate identity for $A_{G\cdot\tau}$ whenever
$(a_\lambda)_{\lambda\in\Lambda}$ 
is an approximate identity for $A_{G\cdot\tau}$,
for every $\tau\in\mathcal{T}_n(G)$.

Fix $\tau\in\mathcal{T}_n(G)$, fix $b\in A_{G\cdot\tau}$ and 
fix an approximate identity 
$(a_\lambda)_{\lambda\in\Lambda}$ of $A_{G\cdot\tau}$. 
For $j,k=0,\ldots,m$ and $h\in H$, the net $\left(\alpha_{x_j^{-1}hx_k}(\pi_k(a_\lambda))\right)_{\lambda\in\Lambda}$ is an approximate identity for 
$A_{x_j^{-1}\tau}$. Using this and part~(1) of \autoref{rem:QuotientMapsAGtau} at the second step, and using the identity $b=\sum_{j=0}^m \pi_j(b)$ at 
the third step, we get
\begin{align*}
\lim_{\lambda\in\Lambda}E_\tau(a_\lambda)b&= \frac{1}{|H|(m+1)}\sum_{j,k=0}^m \sum_{h\in H} \lim_{\lambda\in\Lambda}\alpha_{x_j^{-1}hx_k}(\pi_k(a_\lambda))b\\
&= \frac{1}{|H|(m+1)}\sum_{j,k=0}^m \sum_{h\in H} \pi_j(b)
= \frac{1}{|H|(m+1)}\sum_{k=0}^m \sum_{h\in H} b
=b.
\end{align*}
Analogously, one shows that $\lim_{\lambda\in\Lambda}bE_\tau(a_\lambda)=b$. 
It follows that $(E_\tau(a_\lambda))_{\lambda\in\Lambda}$
is an approximate identity for $A_{G\cdot\tau}$, and the proof is finished.
\end{proof}

Fixed point algebras of decomposable partial actions 
can be explicitly described using certain global subsystems of finite subgroups
of $G$, as we show below.

\begin{thm}\label{prop:FixedPtAlg}
Let $G$ be a discrete group, let $A$ be a \ca, let $n\in\N$, and let 
$\alpha=((A_g)_{g\in G}, (\alpha_g)_{g\in G})$ be a partial action of $G$ on $A$ with the
$n$-decomposition property. 
For every $\tau\in\mathcal{T}_n(G)$, the fixed point algebra of $A_{G\cdot \tau}$
can be canonically identified with $A_\tau^{H_\tau}$, via the restriction of 
$\pi^\tau_0$ to $A_{G\cdot\tau}^G$. Under this identification,
the canonical inclusion 
$\iota_\tau\colon A_\tau^{H_\tau}\hookrightarrow A_{G\cdot \tau}$
is given by $\iota_\tau(a)=\sum_{j=0}^{m} \alpha_{x^\tau_j}^{-1}(a)$
for all $a\in A_{\tau}^{H_\tau}$. 
In particular, 
$A^G$ can be canonically identified with 
$\bigoplus\limits_{z\in \mathcal{O}_n(G)} A_{\tau_z}^{H_z}$.
\end{thm}
\begin{proof}
By \autoref{rem:EquivDecomp}, it suffices to prove the first statement.
Fix $\tau\in\mathcal{T}_n(G)$; we make the abbreviations
$m=m_\tau$, $H_\tau=H$, $\pi_j^\tau=\pi_j$ and $x_j^\tau=x_j$ for $j=0,\ldots,m$. 

Let $\pi_0\colon A_{G\cdot\tau}\to A_\tau$ be the canonical quotient map
described in \autoref{rem:QuotientMapsAGtau}. By part~(1) 
of \autoref{prop:CondExp}, $\pi_0$ restricts to 
a homomorphism $A_{G\cdot \tau}^G\to A_\tau^H$. We claim that this is an 
isomorphism.
To show surjectivity, let $a\in A_{\tau}^{H}$. 
By part~(1) of \autoref{prop:CondExp}, the element
$\varphi_\tau(a)=\frac{1}{|H|}\sum\limits_{j=0}^m  \sum\limits_{h\in H} \alpha_{hx_j}^{-1}(a)$ is $G$-invariant, and we have
\[
\pi_0(\varphi_\tau(a))=\frac{1}{|H|}\sum\limits_{h\in H}\alpha_{h^{-1}}(a)
=a,
\]
as required. 
To show injectivity, let $a\in A_{G\cdot\tau}^{G}$ satisfy $\pi_0(a)=0$. By part~(1) of~\autoref{prop:CondExp}, it follows that $\pi_j(a)=0$ for all $j=1,\ldots,m$, and therefore $a=\sum_{j=0}^{m}\pi_j(a)=0$. 
The remaining claims in the statement are immediate.
\end{proof}

We remark that under the identification $A_{G\cdot\tau}^G\cong A_{\tau}^{H_\tau}$, the canonical conditional expectations $E_\tau$ described in \autoref{prop:CondExp} become the natural conditional expectations $A_{G\cdot\tau}\to A_\tau^H$ given by $a\mapsto \frac{1}{m+1}\sum\limits_{j=0}^{m} \Big(\frac{1}{|H|}\sum\limits_{h\in H} \alpha_{hx_j}(\pi_j(a))\Big)$.


We turn to the existence of a corner embedding 
$c\colon A^G\to A\rtimes_\alpha G$, which is well-known to exist for global
actions of finite groups. 
In the setting of partial actions, 
such a map does not exist in general, even for finite groups.

\begin{eg}\label{eg:CornerMapNotExist}
Let $\alpha$ be the partial action of $\Z_3=\{0,1,2\}$ on $X=[0,1]$ given by 
letting $U=(0,1]$ be the domain of both $\alpha_1$ and $\alpha_2$, 
with $\alpha_1=\alpha_2=\id_{U}$. 
Then $C(X)^G=C(X)$ and 
$C(X)\rtimes_\alpha\Z_3$ contains no nontrivial projections. 
In particular, there is no corner embedding 
$c\colon C(X)^{\Z_3}\to C(X)\rtimes_\alpha \Z_3$.
\end{eg}
\begin{proof}
For $f\in C(X)$, the identity $\alpha_1(fg)=f\alpha_1(g)$ is automatically
satisfied for all $g\in C_0(U)=C(X)_2$, since $\alpha_1=\id_U$. Similarly,
$\alpha_2(fh)=f\alpha_2(h)$ for all $h\in C(X)_1$, and hence $f\in C(X)^G$. 
It follows that $C(X)^G=C(X)$. 
There is a short exact sequence
\[\xymatrix{0\ar[r] & C_0(U)\rtimes_{\alpha|_U}\Z_3\ar[r] & C(X)\rtimes_\alpha\Z_3\ar[r]^-{\pi}& 
 \C\ar[r]& 0,}
\]
where $C_0(U)\rtimes_{\alpha|_U}\Z_3$ is a global crossed product, and $\C$ is identified
with the crossed product of $C(\{0\})$ by the trivial partial action of $\Z_3$.
Since $\alpha_1=\alpha_2=\id_U$, there is an isomorphism $C_0(U)\rtimes_{\alpha|_U}\Z_3\cong C_0(U)\otimes C^*(\Z_3)$. In particular, this algebra has no 
projections other than zero.

Let $p$ be a projection in $C(X)\rtimes_\alpha\Z_3$. Then $\pi(p)$ is a 
projection in $\C$, so it is either 0 or 1. If $\pi(p)=0$, then $p$ belongs
to the kernel of $\pi$, which is $C_0(U)\rtimes_{\alpha|_U}\Z_3$, and hence
is the zero projection. On the other hand, if $\pi(p)=1$, then 
$1-p$ is a projection in the kernel of $\pi$, and hence it must be zero. Thus 
$p=1$. 

To prove the last statement, if a map $c$ as in the statement exists then
it must be an isomorphism, and thus $C(X)\cong C(X)\rtimes_\alpha \Z_3$. 
However, from the above short exact sequence it is clear that the spectrum
of the abelian algebra $C(X)\rtimes_\alpha\Z_3$ contains a point whose complement
has three connected components. In particular, this space is not homeomorphic
to $X=[0,1]$, and the claim follows. 
\end{proof}

For decomposable partial systems, even for infinite groups,
there does in fact exist a canonical corner map 
$c\colon A^G\to A\rtimes_\alpha G$.

\begin{thm}\label{prop:CornerMap}
Let $G$ be a discrete group, let $A$ be a \ca, let $n\in\N$, and let $\alpha$ be a partial 
action of $G$ on $A$ with the $n$-decomposition property.
Fix $\tau\in\mathcal{T}_n(G)$.  
\be
\item There is a corner-embedding $c_\tau\colon A_{G\cdot \tau}^G\to A_{G\cdot \tau}\rtimes_\alpha G$ given on $a\in A_{G\cdot \tau}^G$ by
\[c_\tau(a)=\frac{1}{|H|(m+1)}\sum_{j,k=0}^m\sum_{h\in H} \pi_j(a)\delta_{x_j^{-1}hx_k}.
\]
\item $c_\tau$ is independent of the elements $x_1,\ldots,x_m$.
\item For $g\in \tau^{-1}$, we have $c_{g\tau}=c_\tau$.
\ee
Thus, the map $c=\bigoplus\limits_{z\in \mathcal{O}_n(G)}c_{\tau_z}\colon A^G\to A\rtimes_\alpha G$ is a canonical corner embedding.
\end{thm}
\begin{proof}
(1). Note that $c_\tau$ is well-defined, since 
$\pi_j(a)\in (A_{G\cdot\tau})_{x_j^{-1}hx_k}$
for all $h\in H$ and all $k=0,\ldots,m$, so that the 
product $\pi_j(a)\delta_{x_j^{-1}hx_k}$ is defined in the partial crossed 
product $A_{G\cdot\tau}\rtimes_\alpha G$. 
Moreover, $c_\tau$ is clearly injective, since $c_\tau(a)=0$ if and only if 
$\pi_j(a)=0$ for all $j=0,\ldots,m$, which is equivalent to $a=0$.

\textbf{Claim:} \emph{$c_\tau$ is a homomorphism}. Let $a,b\in A_{G\cdot\tau}^G$.
Use orthogonality of $\alpha_{x_k^{-1}h^{-1}x_j}(\pi_j(a))$ and $\pi_i(a)$ for $k\neq i$ at the third step, and 
part~(1) of \autoref{prop:CondExp} 
at the fifth step, to get
\begin{align*}
c_\tau(a)c_\tau(b)&=\frac{1}{|H|^2(m+1)^2}\sum_{j,k,i,\ell=0}^m\sum_{h, t\in H} \pi_j(a)\delta_{x_j^{-1}hx_k}\pi_i(a)\delta_{x_i^{-1}tx_\ell}\\
&= \frac{1}{|H|^2(m+1)^2}\sum_{j,k,i,\ell=0}^m\sum_{h, t\in H} \alpha_{x_j^{-1}hx_k}(\alpha_{x_k^{-1}h^{-1}x_j}(\pi_j(a))\pi_i(b))\delta_{x_j^{-1}hx_kx_i^{-1}tx_\ell}\\
&=\frac{1}{|H|^2(m+1)^2}\sum_{j,k,\ell=0}^m\sum_{h, t\in H} \alpha_{x_j^{-1}hx_k}(\alpha_{x_k^{-1}h^{-1}x_j}(\pi_j(a))\pi_k(b))\delta_{x_j^{-1}htx_\ell}\\
&=\frac{1}{|H|^2(m+1)^2}\sum_{j,k,\ell=0}^m\sum_{h, t\in H} \pi_j(a)\alpha_{x_j^{-1}hx_k}(\pi_k(b))\delta_{x_j^{-1}htx_\ell}
\end{align*}
\begin{align*}
&=\frac{1}{|H|^2(m+1)^2}\sum_{j,k,\ell=0}^m\sum_{h, t\in H} \pi_j(ab)\delta_{x_j^{-1}htx_\ell}\\
&=\frac{1}{|H|(m+1)}\sum_{j,\ell=0}^m\sum_{h\in H} \pi_j(ab)\delta_{x_j^{-1}hx_\ell}
= c_\tau(ab).\end{align*} 

The rest of the proof consists in proving that $c_\tau(A_{G\cdot \tau}^G)$ is a corner in 
$A_{G\cdot\tau}\rtimes_\alpha G$. 
To this end, we define a multiplier $c_\tau(1)$ of $A_{G\cdot\tau}\rtimes_\alpha G$ 
by setting 
\[c_\tau(1)(a\delta_g)=\frac{1}{|H|(m+1)}\sum\limits_{j,k=0}^{m}\sum\limits_{h\in H}\alpha_{x_j^{-1}hx_k}(\pi_k(a))\delta_{x_j^{-1}hx_kg}\]
for all $a\in (A_{G\cdot\tau})_g$ and all $g\in G$. 
For $j=0,\ldots,m$,
denote by $1_j$ the unit of the multiplier algebra of 
$A_{x_j^{-1}\tau}$. Then the map $c_\tau(1)$ can be identified, in a 
way compatible with the operations in $A_{G\cdot\tau}\rtimes_\alpha G$,
with the formal linear combination
\[c_\tau(1)=\frac{1}{|H|(m+1)}\sum_{j,k=0}^m\sum_{h\in H} 1_j\delta_{x_j^{-1}hx_k}.\]
(When $A_{G\cdot\tau}$, and hence $A_\tau$, is unital, $c_\tau(1)$ is really
the image of $1\in A_{G\cdot\tau}^G$ under $c_\tau$, hence the notation.)
It is easy to check that $c_\tau(1)$ is a multiplier
of $A_{G\cdot\tau}\rtimes_\alpha G$. 

\textbf{Claim:} \emph{$c_\tau(1)$ is a projection}. Let $a\in (A_{G\cdot\tau})_g$. Since $a=\sum_{j=0}^{m}\pi_j(a)$ and $c_\tau$ is linear, it is enough to assume that $a\in A_{x_k^{-1}\tau}\cap A_g$ for some $k=0,\ldots, m$ and show 
$c_\tau(1)(c_\tau(1)(a\delta_g))=c_\tau(1)(a\delta_g)$. 
We may also assume that $g\in x_k^{-1}\tau$, otherwise the identity is satisfied. Let $\ell\in\{0,\ldots,m\}$ and $t\in H$ be the unique elements with $g=x_k^{-1}tx_\ell$. Then
\begin{align*}
c_\tau(1)\left(c_\tau(1)(a\delta_g)\right)&=
\frac{1}{|H|(m+1)}\sum\limits_{j=0}^m\sum\limits_{h\in H} c_\tau(1)\left(\alpha_{x_j^{-1}hx_k}(a)\delta_{x_j^{-1}htx_\ell}\right)\\
&=\frac{1}{|H|^2(m+1)^2}\sum\limits_{i,j=0}^m\sum\limits_{h,s\in H} 
\alpha_{x_i^{-1}sx_j}(\alpha_{x_j^{-1}hx_k}(a))\delta_{x_i^{-1}shtx_\ell}\\
&= \frac{1}{|H|^2(m+1)^2}\sum\limits_{i,j=0}^m\sum\limits_{h,s\in H} 
\alpha_{x_i^{-1}shx_k}(a)\delta_{x_i^{-1}shtx_\ell}\\
&= \frac{1}{|H|(m+1)}\sum\limits_{i=0}^m\sum\limits_{r\in H} 
\alpha_{x_i^{-1}rx_k}(a)\delta_{x_i^{-1}rtx_\ell}
= c_\tau(1)(a\delta_g),
\end{align*}
so $c_\tau(1)=c_\tau(1)^2$. To check that 
$c_\tau(1)$ is self-adjoint, we use its 
presentation as a formal linear combination to get
\begin{align*}
c_\tau(1)^*&=\frac{1}{|H|(m+1)}\sum_{j,k=0}^m\sum_{h\in H} \left(1_j\delta_{x_j^{-1}hx_k}\right)^*
=\frac{1}{|H|(m+1)}\sum_{j,k=0}^m\sum_{h\in H} \left(1_j\delta_{x_j^{-1}hx_k}\right)^*\\
&= \frac{1}{|H|(m+1)}\sum_{j,k=0}^m\sum_{h\in H} 1_k\delta_{x_k^{-1}h^{-1}x_j}
=c_\tau(1).
%
\end{align*}

\textbf{Claim:} \emph{$c_\tau(A_{G\cdot\tau}^G)=c_\tau(1)(A_{G\cdot\tau}\rtimes_\alpha G)c_\tau(1)$.}
Since $c_\tau(1)c_\tau(a)=c_\tau(a)$ for all $a\in A_{G\cdot\tau}^G$, 
it follows that the left-hand side is contained in the right-hand side. 
We prove the converse inclusion. For this, it is enough to show that an 
element of the form $c_\tau(1)(a\delta_g)c_\tau(1)$, for 
$a\in A_{x_k^{-1}\tau}\cap A_g$, belongs to the image of $c_\tau$, for each $k=0,\ldots, m$. This is
immediate if $g\notin x_k^{-1}\tau$, so assume that $g$ has the form 
$g=x_k^{-1}tx_\ell$ for unique $\ell=0,\ldots,m$ and $t\in H$. Note that $1_ia=0$
unless $i=k$, in which case $1_ia=a$. Similarly, $\alpha_{g^{-1}}(a)1_p=0$ unless $p=\ell$,
in which case $\alpha_{g^{-1}}(a)1_p=\alpha_{g^{-1}}(a)$. 
We have
\begin{align*}
c_\tau(1)(a\delta_g)c_\tau(1)
&= \frac{1}{|H|^2(m+1)^2}\sum_{i,j,p,q=0}^m\sum_{h,s\in H} (1_j\delta_{x_j^{-1}hx_i})(a\delta_g)(1_p\delta_{x_p^{-1}sx_q})\\
&= \frac{1}{|H|^2(m+1)^2}\sum_{i,j,p,q=0}^m\sum_{h,s\in H} \left(\alpha_{x_j^{-1}hx_i}(1_ia)\delta_{x_j^{-1}hx_ig}\right)(1_p\delta_{x_p^{-1}sx_q})\\
&= \frac{1}{|H|^2(m+1)^2}\sum_{j,p,q=0}^m\sum_{h,s\in H} \left(\alpha_{x_j^{-1}hx_k}(a)\delta_{x_j^{-1}htx_\ell}\right)(1_p\delta_{x_p^{-1}sx_q})\\
&= \frac{1}{|H|^2(m+1)^2}\sum_{j,p,q=0}^m\sum_{h,s\in H} 
\alpha_{x_j^{-1}htx_\ell}(\alpha_{x_\ell^{-1}t^{-1}h^{-1}x_j}(\alpha_{x_j^{-1}hx_k}(a))1_p)
\delta_{x_j^{-1}htx_\ell x_p^{-1}sx_q}\\
&= \frac{1}{|H|^2(m+1)^2}\sum_{j,p,q=0}^m\sum_{h,s\in H} 
\alpha_{x_j^{-1}htx_\ell}(\alpha_{x_\ell^{-1}t^{-1}x_k}(a)1_p)
\delta_{x_j^{-1}htx_\ell x_p^{-1}sx_q}\\
&= \frac{1}{|H|^2(m+1)^2}\sum_{j,q=0}^m\sum_{h,s\in H} 
\alpha_{x_j^{-1}hx_k}(a)
\delta_{x_j^{-1}htsx_q}
\end{align*}
Set $b=\frac{1}{|H|(m+1)}\sum\limits_{h\in H} 
\sum\limits_{j=0}^m\alpha_{x_j^{-1}hx_k}(a)$.
With $E_\tau$ denoting the canonical conditional expectation constructed
in \autoref{prop:CondExp}, we have $b=E_\tau(a)$, and hence $b$ belongs 
to $A_{G\cdot\tau}^G$. Finally, it is clear that $c_\tau(b)=c_\tau(1)(a\delta_g)c_\tau(1)$, as desired.

(2). We write $c_\tau^x$ for the corner embedding constructed from the 
elements $x_0=1,x_1,\ldots,x_m$ as in part~(1). 
Let $y_0=1,y_1,\ldots,y_m\in G$ be elements satisfying $\tau=\bigsqcup_{j=0}^m Hy_j$, and write $c_\tau^y$ for the corresponding 
corner embedding. We want to show that $c_\tau^x=c_\tau^y$. 
By part~(3) of \autoref{lma:OrbitTuples}, we may assume without loss of 
generality that there exist $h_1,\ldots,h_m\in H$ satisfying
$y_j=h_jx_j$ for all $j=1,\ldots,m$. 
Let $a\in A_{G\cdot \tau}^G$. Then 
$\pi^y_j(a)=\pi^x_j(\alpha_{h_j}^{-1}(a))=\pi_j^x(a)$ for all $j=0,\ldots,m$. 
Using this at the second step, we get
\begin{align*}
c^y_\tau(a)&=\frac{1}{|H|(m+1)}\sum_{j,k=0}^m\sum_{h\in H} \pi^y_j(a)\delta_{y_j^{-1}hy_k}\\
&=\frac{1}{|H|(m+1)}\sum_{j,k=0}^m\sum_{h\in H} \pi^x_j(a)\delta_{x_j^{-1}h_j^{-1}hh_kx_k}=c_\tau^x(a).
\end{align*}

(3). Let $g\in \tau^{-1}$, and let $\ell=0,\ldots,m$ be the 
unique element satisfying $g\in x_\ell^{-1}H$. Then $g\tau=x_\ell^{-1}\tau$. In
particular, it suffices to assume that $g=x_\ell^{-1}$. 
Set $\sigma=x_\ell^{-1}\tau$.
By part~(2) of this theorem, we may compute $c_{\sigma}$ using any 
decomposition 
$\sigma=H_\sigma\sqcup H_\sigma y_1\sqcup \ldots \sqcup H_\sigma y_m$ 
as in \autoref{lma:OrbitTuples}. In this context, we must have 
$H_\sigma=x_\ell^{-1}H_\tau x_\ell$, and we take $y_j=x_\ell^{-1}x_j$ for 
$j=0,\ldots,m$. Note that $\pi_j^\tau=\pi_j^\sigma$ for all $j=0,\ldots,m$.
Given $a\in A_{G\cdot\tau}^G$, we have
\begin{align*}
c_{\sigma}(a)&= \frac{1}{|H_\sigma|(m+1)}\sum_{j,k=0}^m\sum_{t\in H_\sigma}
\pi^\sigma_j(a)\delta_{y_j^{-1}ty_k}\\
&= \frac{1}{|H_\tau|(m+1)}\sum_{j,k=0}^m\sum_{h\in H_\tau}
\pi^\tau_j(a)\delta_{x_j^{-1}x_\ell (x_\ell^{-1}hx_\ell)x_\ell^{-1}x_k}=c_{\tau}(a).
\end{align*}
The last statement of the theorem follows from part~(3) of
\autoref{prop:EquivDecomp}.
\end{proof}

\section{Computation of the crossed product}
The crossed product of 
a decomposable action was computed, up to Morita equivalence,
in \autoref{cor:CPuptoMorEquiv}.
For some purposes, such as \autoref{thm:Freeness}, it is necessary to have a 
computation of $A\rtimes_\alpha G$ up to isomorphism and not just up to
Morita equivalence. Obtaining such a description is the goal of this section;
see \autoref{prop:CrossedProductTupleOrb}. 
Using this calculation, we provide 
an alternative computation of the partial group algebra $C^*_{\mathrm{par}}(G)$ of a finite group $G$ from \cite{DokExePic_partial_2000}; see \autoref{thm:PartialGpAlg}. Finally, in \autoref{thm:Freeness} we combine the 
results from Section~4 and \autoref{prop:CrossedProductTupleOrb} to give a
characterization of freeness for decomposable topological partial actions, in terms of the corner map $c$ defined in \autoref{prop:CornerMap}. 
This characterization fails in general, even for free partial actions of finite
groups; see \autoref{eg:MoritaEqFails}.

\begin{thm}\label{prop:CrossedProductTupleOrb}
Let $G$ be a discrete group, let $A$ be a \ca, let $n\in\N$, and let 
$\alpha=((A_g)_{g\in G}, (\alpha_g)_{g\in G})$ be a partial action of $G$ on $A$ with the
$n$-decomposition property. 
For every $\tau\in \mathcal{T}_n(G)$, there is a natural isomorphism
\[\psi_\tau \colon A_{G\cdot \tau}\rtimes_{\alpha} G\to M_{m_\tau+1}(A_\tau\rtimes_{\alpha|_{H_\tau}} H_\tau)\]
which satisfies $\psi_\tau(a\delta_{x_j^{-1}hx_k})=\alpha_{x_j}(a)v_h\otimes e_{j,k}$ for all $a\in A_{x_j^{-1}\tau}$, for all $j,k=0,\ldots,m$, and all 
$h\in H$, where $v\colon H\to M(A_\tau\rtimes_{\alpha|_{H}} H)$ denotes the canonical unitary representation.
It follows that $A\rtimes_\alpha G$ can be naturally identified with 
$\bigoplus_{z\in\mathcal{O}_n(G)}M_{m_z+1}(A_z\rtimes_{\alpha|_{H_z}} H_z)$.
\end{thm}
\begin{proof}
By \autoref{rem:EquivDecomp}, it suffices to prove the first statement.
Fix $\tau\in\mathcal{T}_n(G)$; we use \autoref{nota:Section} and 
\autoref{rem:QuotientMapsAGtau}.
 Define maps $\varphi\colon A_{G\cdot\tau} \to M_{m+1}(A_\tau\rtimes_{\alpha|_H}H)$ and $u\colon G \to M(M_{m+1}(A_\tau\rtimes_{\alpha|_H}H))\cong M(A_\tau\rtimes_{\alpha|_H}H)\otimes M_{m+1}$ by
\[\varphi(a)=\sum\limits_{j=0}^{m} \alpha_{x_j}(\pi_j(a))\otimes e_{j,j}
 \ \mbox{ and } \ 
 u_g=\sum\limits_{j,k=0}^{m} \mathbbm{1}_{\{g\in x_j^{-1}Hx_k \}} v_{x_jgx_k^{-1}}\otimes e_{j,k}\]
for $a\in A_{G\cdot\tau}$ and $g\in G$. We will see that
$(\varphi, u)$ is a covariant pair for $(A_{G\cdot\tau},\alpha)$.

\textbf{Claim:} \emph{$u$ is a partial representation of $G$.}
To check this, let $g\in G$. We have

\begin{align*}
u_g^*&= \Big(\sum\limits_{j,k=0}^{m} \mathbbm{1}_{\{g\in x_j^{-1}Hx_k \}} v_{x_jgx_k^{-1}}\otimes e_{j,k}\Big)^*\\
&=\sum\limits_{j,k=0}^{m} \mathbbm{1}_{\{g^{-1}\in x_k^{-1}Hx_j \}} v_{x_kg^{-1}x_j^{-1}}\otimes e_{k,j}=u_{g^{-1}}.
\end{align*} 
Moreover, for $g_1,g_2\in G$,
the product $u_{g_1}u_{g_2}u_{g_2^{-1}}$ is equal to:
\begin{align*}
&\Big(\sum\limits_{j,k=0}^{m} \mathbbm{1}_{\{g_1\in x_j^{-1}Hx_k \}} v_{x_jg_1x_k^{-1}}\otimes e_{j,k}\Big)
 \cdot\Big(\sum\limits_{\ell,r=0}^{m} \mathbbm{1}_{\{g_2\in x_\ell^{-1}Hx_r \}} v_{x_\ell g_2x_r^{-1}}\otimes e_{\ell, r}\Big)\\
 & \ \ \ \ \  \ \ \ \ \ \ \ \ \ \ \ \ \ \  \ \ \ \ \ \ \ \ \ \cdot\Big(\sum\limits_{s,t=0}^{m} \mathbbm{1}_{\{g_2^{-1}\in x_s^{-1}Hx_t \}} v_{x_sg_2^{-1}x_t^{-1}}\otimes e_{s,t} \Big)\\
&=\sum\limits_{j,k,r,t=0}^{m} \mathbbm{1}_{\{g_1\in x_j^{-1}Hx_k  \}} \mathbbm{1}_{\{g_2\in x_k^{-1}Hx_r \}}\mathbbm{1}_{\{g_2^{-1}\in x_r^{-1}Hx_t \}}
v_{x_jg_1x_k^{-1}} v_{x_kg_2x_r^{-1}}v_{x_rg_2^{-1}x_t^{-1}}\otimes e_{j,t}\\
&=\sum\limits_{j,k,r=0}^{m} \mathbbm{1}_{\{g_1\in x_j^{-1}Hx_k  \}} \mathbbm{1}_{\{g_2\in x_k^{-1}Hx_r \}}v_{x_jg_1x_k^{-1}}\otimes e_{j,k},
\end{align*}
where at the last step, we use that $x_k^{-1}H\cap x_t^{-1}H=\emptyset$ for $k\neq t$, which implies that
$\mathbbm{1}_{\{g_2\in x_k^{-1}Hx_r \}}\mathbbm{1}_{\{g_2^{-1}\in x_r^{-1}Hx_t \}}=0$ for $k\neq t$ and $r=0,\ldots, m$. On the other hand, 
\begin{align*}
u_{g_1g_2}u_{g_2^{-1}}&=\sum\limits_{j,k,r=0}^{m} \mathbbm{1}_{\{g_1g_2\in x_j^{-1}Hx_r \}}\mathbbm{1}_{\{g_2^{-1}\in x_r^{-1}Hx_k\}} v_{x_jg_1g_2x_r^{-1}}v_{x_rg_2^{-1}x_k^{-1}}\otimes e_{j,k}\\
&=\sum\limits_{j,k,r=0}^{m} \mathbbm{1}_{\{g_1g_2\in x_j^{-1}Hx_r \}}\mathbbm{1}_{\{g_2\in x_k^{-1}Hx_r\}} v_{x_jg_1x_k^{-1}}\otimes e_{j,k}.
\end{align*}

\textbf{Claim:}\emph{ the pair $(\varphi,u)$ is a covariant representation for $(A_{G\cdot \tau},\alpha)$.}
To prove the claim, fix $g\in G$. For $a\in (A_{G\cdot\tau})_{g^{-1}}$, we must show that 
$u_g\varphi(a)u_{g^{-1}}=\varphi(\alpha_g(a))$.
By linearity and the decomposition property, we may assume that there are $k\in\{0,\ldots,m\}$ and $g^{-1}\in x_k^{-1}\tau$ such that $a\in A_{x_k^{-1}\tau}$. There exist unique $j\in\{0,\ldots, m\}$ and $h\in H$ such that $g=x_j^{-1}hx_k$.
Since
$\alpha_g(a)$ belongs to $A_{x_j^{-1}\tau}$, we get
\[\varphi(a)=\alpha_{x_k}(a)\otimes e_{k,k} \ \mbox{ and } \ \varphi(\alpha_g(a))=\alpha_{x_j}(\alpha_g(a))\otimes e_{j,j}.\]
Using at the first step the uniqueness of $j$, we conclude that
\begin{align*}
u_g\varphi(a)u_{g^{-1}}&=(v_h\otimes e_{j,k})(\alpha_{x_k}(a)\otimes e_{k,k})(v_h^*\otimes e_{k,j})\\
&=\alpha_{hx_k}(a)\otimes e_{j,j}
=\alpha_{x_jg}(a)\otimes e_{j,j}=\varphi(\alpha_g(a)),
\end{align*}
and the claim is proved.

By the universal property of the partial crossed product, 
there is a homomorphism
\[\psi_\tau\colon A_{G\cdot \tau}\rtimes G\to M_{m+1}(A_\tau\rtimes_{\alpha|_H}H).\]
satisfying $(\psi_\tau)(a\delta_g)=\varphi(a)u_g$ whenever $g\in G$ and $a\in (A_{G\cdot \tau})_{g}$. It is clear that $\psi_\tau$ satisfies the formula 
given in the statement.

\textbf{Claim:} \emph{the map $\psi_\tau$ is an isomorphism.} We construct an inverse map to $\psi_\tau$. Since $H$ is a finite group, the $C^*$-algebra $M_{m+1}(A_\tau\rtimes_{\alpha|_{H}}H)$ is linearly spanned by elements of the form $av_h\otimes e_{j,k}$, for  $a\in A_\tau$, $h\in H$, and $j,k\in\{0,\ldots,m\}$.
Define a linear map \[\phi_\tau: M_{m+1}(A_\tau\rtimes_{\alpha|_{H}}H)\to A_{G\cdot\tau}\rtimes_\alpha G,\] 
by setting $\phi_\tau(av_h\otimes e_{j,k})=\alpha_{x_j^{-1}}(a)\delta_{x_j^{-1}hx_k}$, for all $a\in A_\tau$, for all $j,k=0,\ldots,m$, and all $h\in H$. It is easily seen that $\phi_\tau\circ\psi_\tau=\id_{A_{G\cdot\tau}\rtimes_\alpha G}$ and $\psi_\tau\circ\phi_\tau=\id_{M_{m+1}(A_\tau\rtimes_{\alpha|_H}H)}$,
by checking these identities on (linear) generators. We conclude
that $\psi_\tau$ is bijective, and hence an isomorphism.
\end{proof}

As a first application of \autoref{prop:CrossedProductTupleOrb}, 
give an alternative proof for the explicit description of the 
partial group algebra of a finite group from \cite{DokExePic_partial_2000}. 

\begin{thm}\label{thm:PartialGpAlg}
Let $G$ be a finite group. Then there is a canonical identification
\[C^*_{\mathrm{par}}(G)\cong \bigoplus_{n=1}^{|G|}\bigoplus_{z\in \mathcal{O}_n(G)}
 M_{m_z+1}(C^*(H_z)).\]
\end{thm}
\begin{proof} 
Set $\mathcal{T}(G)=\bigsqcup_{n=1}^{|G|} \mathcal{T}_n(G)$, endowed with
its canonical partial action of $G$.
Recall from Section~6 in~\cite{Exe_partial_1998} that $C^*_{\mathrm{par}}(G)$
can be canonically identified with $C(\mathcal{T}(G))\rtimes G$. On the other 
hand, we have 
$C(\mathcal{T}(G))\rtimes G\cong \bigoplus_{n=1}^{|G|} C(\mathcal{T}_n(G))\rtimes G$,
and for $n=1,\ldots,|G|$, 
the canonical action of $G$ on $C(\mathcal{T}_n(G))$ has the 
$n$-decomposition property by \autoref{prop:TnGhasInterProp}. It follows
from \autoref{prop:CrossedProductTupleOrb} that
\begin{align*} 
C^*_{\mathrm{par}}(G)&\cong \bigoplus_{n=1}^{|G|}\bigoplus_{z\in \mathcal{O}_n(G)}
 M_{m_z+1}(C(\{z\})\rtimes H_z )\cong 
 \bigoplus_{n=1}^{|G|}\bigoplus_{z\in \mathcal{O}_n(G)}
 M_{m_z+1}(C^*(H_z)).\qedhere
 \end{align*}
\end{proof}

We close this section with a second application, this time to 
topological partial actions; see \autoref{thm:Freeness}.
It is well-known that for a finite group action $G\curvearrowright X$ on a locally compact Hausdorff space $X$, freeness is equivalent to the canonical
corner map $c\colon C_0(X)^G\to C_0(X)\rtimes G$ having full range\footnote{This
means that $c(C_0(X)^G)$ is a full corner in $C_0(X)\rtimes G$.}.
(In particular, $C_0(X)^G$ is Morita equivalent to $C_0(X)\rtimes G$.)
This result fails in general for partial actions of finite groups, 
even if one uses $C_0(X/G)$ instead of $C_0(X)^G$. 

Recall that a partial action $\sigma$ of a group $G$ on a topological space $X$
is said to be \emph{free} if whenever $g\in G$ and $x\in X$ satisfy 
$\sigma_g(x)=x$, then $g=1$.

\begin{eg}\label{eg:MoritaEqFails}
Set $X=(0,2]$ and $U=(0,1)\cup (1,2)\subseteq X$, and let $\sigma\colon U\to U$ be the order-2 homeomorphism given by $\sigma(x)=\begin{cases*}
      x+1 & if $x<1$ \\
      x-1 & if $x>1$.
    \end{cases*}$ for $x\in U$. Denote by $\alpha$ the partial action of $G=\Z_2$ on $X$ determined by $\sigma$
(also used in \autoref{eg:NoCondExp}). Then $\alpha$ is free, and 
$C_0(X)\rtimes_\alpha G$ is not Morita equivalent to either $C_0(X)^G$ or 
$C(X/G)$.
\end{eg}
\begin{proof}
That $\alpha$ is free is clear, since $\alpha_1(x)=x$ implies that $x$
belongs to $U$ and $\sigma(x)=x$, which is not possible. Note that
$C_0(X)^G\cong C_0((0,1))$ (see \autoref{eg:NoCondExp}). Moreover,  
$X/G$  can be naturally identified as a set with $[1,2]$. The topology 
on $[1,2]$ induced by this identification is the usual one when 
restricted to $(1,2]$, while a neighborhood base at $\{1\}$ is given by sets of the form $[1,1+\delta)\cup (2-\delta,2)$ for $\delta>0$. This is not a Hausdorff topology since 1 and 2 cannot be separated, but $C(X/G)$ can be identified with $C(S^1)$, since any continuous function on $[1,2]$ must 
take the same value at 1 and 2. 

We show that 
the $K$-groups of $C_0(X)\rtimes_\alpha G$ are not isomorphic to those of
either $C_0(X)^{G}$ or 
$C(X/G)$, from which it will follow that
it is not Morita equivalent to either of them. Note that $K_1(C_0(X)^{G})\cong K_1(C(X/G))\cong \Z$.
There is a short exact sequence
\[\xymatrix{0\ar[r]& C_0(U)\rtimes_\sigma G\ar[r]& C_0(X)\rtimes_\alpha G\ar[r] & C(\{0,1\})\rtimes_{\mathrm{trivial}}G\ar[r]&0,}\]
where the partial action of $G$ on $\{0,1\}$ is the trivial one (with 
trivial domains). Its crossed product is thus $\C\oplus \C$. 
Morever, since $U$ is equivariantly isomorphic to $(0,2)\times G$ with 
$\sigma$ corresponding to the product of the identity on $(0,2)$ and the 
left translation action on $G$, it follows that 
$C_0(U)\rtimes_\sigma G\cong C_0((0,2))\otimes M_2$. The induced six-term 
exact sequence in $K$-theory takes the following form:
\[
\xymatrix{0\ar[r] & K_0(C_0(X)\rtimes_\alpha G)\ar[r] & \Z\oplus\Z \ar[d]\\
0\ar[u]& K_1(C_0(X)\rtimes_\alpha G)\ar[l] & \Z. \ar[l]
}
\]
If $K_1(C_0(X)\rtimes_\alpha G)\cong \Z$, then the bottom-right map is 
an isomorphism, thus the vertical-right map is zero, and thus 
$K_0(C_0(X)\rtimes_\alpha G)\cong \Z\oplus\Z$, which is not isomorphic
to the $K_0$-group of either $C_0(X)^{G}$ or $C(X/G)$. The proof is finished.
\end{proof}

Having identified the crossed product $A_{G\cdot\tau}\rtimes_\alpha G$ with
$M_{m+1}(A_\tau\rtimes_\alpha H)$ in the previous theorem, and having identified $A_{G\cdot \tau}^G$ with 
$A_\tau^H$ in \autoref{prop:FixedPtAlg}, 
in the next lemma we give a description of the corner map 
$c_\tau\colon A_{G\cdot\tau}^G\to A_{G\cdot\tau}\rtimes_\alpha G$ from 
\autoref{prop:CornerMap}. 

\begin{lma}\label{lma:IdentifyCornerMap}
Let $G$ be a discrete group, let $A$ be a \ca, let $n\in\N$, and let 
$\alpha=((A_g)_{g\in G}, (\alpha_g)_{g\in G})$ be a partial action of $G$ on $A$ with the
$n$-decomposition property. 
For $\tau\in \mathcal{T}_n(G)$, let 
$\psi_\tau\colon A_{G\cdot\tau}\rtimes_\alpha G\to 
M_{m+1}(A_\tau\rtimes_\alpha H)$ and $\phi_\tau \colon A_{G\cdot\tau}^G\to A_\tau^H$ be the isomorphisms from
\autoref{prop:CrossedProductTupleOrb} and \autoref{prop:FixedPtAlg}, respectively. Let 
\[c_\tau\colon A_{G\cdot\tau}^G\to A_{G\cdot\tau}\rtimes_\alpha G
\mbox{ and } c_H\colon A_\tau^H\to A_\tau\rtimes_\alpha H
\]
be the canonical corner embeddings described in \autoref{prop:CornerMap}.
(Note that $H\curvearrowright A_\tau$ has the $|H|$-decomposition property,
since it is a global action.) Denote by 
$e\in M_{m+1}$ the rank-one projection given by $e=\frac{1}{m+1}\sum_{j,k=0}^m e_{j,k}$.
Then the composition
\[ \psi_\tau\circ c_\tau\circ \phi_\tau^{-1}\colon A_\tau^H\to M_{m+1}(A_\tau\rtimes_\alpha H)\]
is given by $(\psi_\tau\circ c_\tau\circ \phi_\tau^{-1})(a)=c_H(a)\otimes e$
for all $a\in A_\tau^H$.
\end{lma}
\begin{proof}
Let $a\in A_\tau^H$. Then $\phi_\tau^{-1}(a)=\sum_{j=0}^m \alpha^{-1}_{x_j}(a)$.
Then 
\[c_\tau(\phi_\tau^{-1}(a))=\frac{1}{|H|(m+1)}\sum_{j,k=0}^m \sum_{h\in H}\alpha_{x_j}^{-1}(a)\delta_{x_j^{-1}hx_k},\]
and applying $\psi_\tau$ to the above expression gives 
\begin{align*}(\psi_\tau\circ c_\tau\circ \phi_\tau^{-1})(a)&=
 \frac{1}{|H|(m+1)}\sum_{j,k=0}^m \sum_{h\in H}(a\otimes e_{j,j})(v_h\otimes e_{j,k})\\
 &=\frac{1}{|H|(m+1)}\sum_{j,k=0}^m \sum_{h\in H}av_h\otimes e_{j,k}\\
 &= \Big(\frac{1}{|H|}\sum_{h\in H}av_h\Big)\otimes \Big(\frac{1}{m+1}\sum_{j,k=0}^m e_{j,k}\Big)
 = c_H(a)\otimes e.\qedhere
\end{align*}
\end{proof}

In the next theorem we obtain the desired characterization of freeness
in terms of the map $c$. We also show that for partial actions with the
decomposition property, $C_0(X)^G$ is canonically isomorphic to $C_0(X/G)$,
while we have seen that this fails in general for partial actions even of
finite groups; see \autoref{eg:MoritaEqFails}.

\begin{thm}\label{thm:Freeness}
Let $G$ be a discrete group, let $X$ be a locally compact Hausdorff space, 
and let 
$\sigma=((X_g)_{g\in G}, (\sigma_g)_{g\in G})$ be a partial action with the
decomposition property. Then:
\be\item There is a natural isomorphism
$C_0(X)^G\cong C_0(X/G)$.
\item $\sigma$ is free if and only if the corner 
embedding $c\colon C_0(X)^G\to C_0(X)\rtimes_\sigma G$ from \autoref{prop:CornerMap}
has full range. 
\ee\end{thm}
\begin{proof}
Let $n\in\N$ be such that $\alpha$ has the $n$-decomposition property.
For $\tau\in\mathcal{T}_n(G)$, we set $X_\tau=\bigcap_{g\in \tau}X_g$
and $X_{G\cdot \tau}=\bigcup_{g\in \tau}X_{g^{-1}\tau}$.
By the $n$-decomposition property, the sets $X_\tau$ are pairwise
disjoint, and there is an equivariant disjoint-union decomposition 
$X=\bigsqcup_{z\in\mathcal{O}_n(G)} X_{G\cdot \tau_z}$. 
We adopt the conventions from \autoref{nota:Section}.

(1). It suffices to prove the statement for the restriction of $\sigma$
to $X_{G\cdot\tau}$. 
We define a function $\varphi\colon X_\tau/H\to X_{G\cdot \tau}/G$ given by 
$\varphi(\mathrm{orb}_H(x))=\mathrm{orb}_G(x)$ for all $x\in X_\tau$.
Note that $\varphi$ is well-defined. 
We claim that $\varphi$ is a homeomorphism. 

Let $x,y\in X_\tau$ satisfy $\mathrm{orb}_G(x)=\mathrm{orb}_G(y)$, and let 
$g\in G$ satisfy $g\cdot x=y$. Then $g\cdot \tau=\tau$, and thus
$g\in H$. We conclude that $\mathrm{orb}_H(x)=\mathrm{orb}_H(y)$ 
as desired. 
To check surjectivity, let $x\in X_{G\cdot \tau}$. Choose $j=0,\ldots,m$
such that $x\in X_{x_j^{-1}\tau}$. Then $x_j\cdot x$ belongs to $X_\tau$, and 
\[\mathrm{orb}_G(x)=\mathrm{orb}_G(x_jx)=\varphi(\mathrm{orb}_H(x_jx)),\]
as desired. 
Finally, we show that $\varphi$ is open. Consider the commutative diagram
\begin{align*}
\xymatrix{
X_\tau \ar@{^{(}->}[r]\ar[d]_-{\pi_H}& X_{G\cdot\tau}\ar[d]^-{\pi_G}\\
X_\tau/H \ar[r]_-{\varphi} & X_{G\cdot \tau}/G,
}
\end{align*}
where $\pi_H$ and $\pi_G$ are the canonical quotient maps. Both these 
maps are open and finite-to-one: this is clear for $\pi_H$, while for $\pi_G$
it follows from the fact that the orbits of $G\curvearrowright X_{G\cdot\tau}$
are finite. Since the inclusion $X_\tau\hookrightarrow X_{G\cdot\tau}$ is open,
we deduce that $\varphi$ is open, as desired.

Using \autoref{prop:FixedPtAlg} at the first step; using that $H\curvearrowright X_\tau$ is a global action of a finite group at the second step; and using
the above claim at the third step, we obtain the following natural isomorphisms:
\[C_0(X_{G\cdot \tau})^G\cong C_0(X_\tau)^H \cong C_0(X_\tau/H)\cong 
 C_0(X_{G\cdot\tau}/G).
\]

(2). We begin by proving that $\sigma$ is free if and only if 
$H_\tau\curvearrowright^{\sigma|_{H_\tau}} X_\tau$ is free for every 
$\tau\in\mathcal{T}_n(G)$.
By the preceding comments, it suffices to show, for a fixed 
$\tau\in\mathcal{T}_n(G)$, that the partial action of $G$ on 
$X_{G\cdot \tau}$ is free if and only if $H_\tau \curvearrowright X_\tau$ is free. 
Note that the ``only if'' implication is immediate. 
Conversely, assume that $H_\tau\curvearrowright X_\tau$ is 
free, and let $g\in G$ and $x\in X_{G\cdot\tau}\cap X_{g}
=\bigsqcup_{0\leq j\leq m\colon g\in x_j^{-1}\tau}X_{x_j^{-1}\tau}$ satisfy $\sigma_{g^{-1}}(x)=x$. There exists a unique $j\in\{0,\ldots,m\}$ such that $x\in X_{x_j^{-1}\tau}$ and $g\in x_j^{-1}\tau$.
As $x_j^{-1}\tau=\bigsqcup_{k=0}^{m}x_j^{-1}Hx_k$, we let $k\in \{0,\ldots,m\}$ and $h\in H$ be the unique elements such that $g= x_j^{-1}hx_k$. 
Then $x=\sigma_{g^{-1}}(x)\in\sigma_{g^{-1}}(X_{x_j^{-1}\tau})\subseteq X_{x_k^{-1}\tau}$. It follows that $j=k$. Set 
$y=\sigma_{x_j}(x) \in X_\tau$. Then
\[\sigma_h(y)=\sigma_{hx_j}(x)=\sigma_{x_j}\big(\sigma_{x_j^{-1}hx_j}(x)\big)=\sigma_{x_j}(x)=y.\]
Since the action $H_\tau\curvearrowright X_\tau$ is free, it follows that 
$h=1$ and hence $g=x_j^{-1}hx_j=1$.

We turn to the statement of the theorem. We have just shown that $\sigma$ is
free if and only if $H_\tau\curvearrowright X_\tau$ is free for every 
$\tau\in\mathcal{T}_n(G)$. Since these are global actions of finite groups on locally
compact Hausdorff spaces, this is in turn equivalent to the canonical
corner map
$C_0(X_\tau)^{H_\tau}\to C_0(X_\tau)\rtimes H_\tau$ having full range 
for every $\tau\in\mathcal{T}_n(G)$ (see 
Proposition~7.1.12 and Theorem~7.2.6 in~\cite{Phi_equivariant_1987}, noting that saturation
means precisely that the corner embedding from the fixed
point algebra into the crossed product has full range).
By \autoref{lma:IdentifyCornerMap}, this is itself equivalent to the canonical corner map 
$c_\tau\colon C_0(X_{G\cdot \tau})^G\to C_0(X_{G\cdot \tau})\rtimes_\alpha G$ 
from \autoref{prop:CornerMap} having full range, for all 
$\tau\in\mathcal{T}_n(G)$. In turn, this is equivalent to the canonical
corner map $c\colon C_0(X)^G\to C_0(X)\rtimes_\alpha G$ having full range,
as desired.
\end{proof}

\section{Decompositions of partial actions of finite groups}
In this final section, we show that \emph{any} partial action of a 
finite group can be written as
an iterated extension of decomposable partial actions;
see \autoref{thm:ExtensionIntersProp}.
This allows us to reduce the computation of such partial 
crossed products to an extension problem for global crossed products. 
As an application, 
we prove that crossed products of partial actions of finite
groups preserve the property of having finite 
stable rank; see \autoref{cor:FiniteGpCParbitraryPreserve}.


\begin{thm}\label{thm:ExtensionIntersProp}
Let $G$ be a finite group, let $A$ be a \ca, and let 
$\alpha$ be a partial action of $G$ on $A$. 
Then there are canonical equivariant extensions 
\[\xymatrix{0\ar[r] & (D^{(k)},\delta^{(k)}) \ar[r]& (A^{(k)},\alpha^{(k)})\ar[r] & (A^{(k-1)},\alpha^{(k-1)})\ar[r]&0,}
\]
for $2\leq k \leq |G|$, with $(A^{(|G|)},\alpha^{(|G|)})=(A,\alpha)$ and satisfying the following properties 
\be
\item[(a)] $\delta^{(k)}$ has the $k$-decomposition property;
\item[(b)] $A_\sigma^{(k)}=\{0\}$ for all $\sigma\in \mathcal{T}_{k+1}(G)$;
\item[(c)] $\alpha^{(1)}$ has the 1-decomposition property.
\ee
Thus $\alpha$ can be 
written canonically as an iterated extension of 
decomposable
partial actions.
\end{thm}
\begin{proof}
Set $n=|G|$ and take $(A^{(n)},\alpha^{(n)})=(A,\alpha)$.
Define $D^{(n)}=\bigcap_{g\in G} A^{(n)}_g$, which is a $G$-invariant 
ideal in $A^{(n)}$. The 
restriction $\delta^{(n)}$ of $\alpha^{(n)}$ to $D^{(n)}$ is a 
global action. By item~(2) in~\autoref{eg:IntProp}, we deduce that
$\delta^{(n)}$ has the $n$-decomposition property. Denote 
by $A^{(n-1)}$ the quotient of $A^{(n)}$ by $D^{(n)}$, and 
let 
$\pi^{(n)}\colon A^{(n)}\to A^{(n-1)}$ be the canonical equivariant
quotient map. 
Then 
\[\bigcap_{g\in G}A_g^{(n-1)}=\pi^{(n)}\Big(\bigcap_{g\in G}A_g^{(n)}\Big)
=\pi^{(n)}(D^{(n)})=0.\]
In other words, $A_\sigma^{(n-1)}=\{0\}$ for all $\sigma\in \mathcal{T}_{n}(G)=\{G\}$.
Assuming that $(A^{(k)},\alpha^{(k)})$ has been constructed and satisfies
condition (b) above, we will construct 
$(D^{(k)},\delta^{(k)})$ and $(A^{(k-1)},\alpha^{(k-1)})$. We set 
$D^{(k)}=\sum_{\tau\in \mathcal{T}_k(G)}A^{(k)}_\tau$, which is 
a $G$-invariant ideal in $A^{(k)}$. Let $\delta^{(k)}$ be the induced
action on it; we claim that $\delta^{(k)}$ has the $k$-decomposition
property. 

For $g\in G$, we have $D^{(k)}_g=\sum_{\tau\in \mathcal{T}_k(G)_g}A^{(k)}_\tau$, and hence $D^{(k)}_\tau=A^{(k)}_\tau$ for all $\tau\in\mathcal{T}_k(G)$. 
It follows that $D^{(k)}$ satisfies condition~(a) in \autoref{df:nIntProp}. In 
order to verify condition~(b), let $\tau\in\mathcal{T}_k(G)$ and let 
$g\notin\tau$. Set $\sigma=\tau\cup\{g\}$, which is a tuple in 
$\mathcal{T}_{k+1}(G)$. Using condition~(b) of the inductive step at the first
step, we get
\[\{0\}=A^{(k)}_\sigma=A_\tau^{(k)}\cap A_g^{(k)}=D_\tau^{(k)}\cap A_g^{(k)}.\]
Since $D_g^{(k)}\subseteq A_g^{(k)}$, it follows that 
$D_\tau^{(k)}\cap D_g^{(k)}=\{0\}$ as desired. Thus $\delta^{(k)}$ has
the $k$-decomposition property.
Let
$A^{(k-1)}$ denote the quotient of $A^{(k)}$ by $D^{(k)}$,
and let $\pi^{(k)}\colon A^{(k)}\to A^{(k-1)}$ be the canonical equivariant
quotient map. 
Then 
\[\sum_{\tau\in\mathcal{T}_k(G)}A_\tau^{(k-1)}=\pi^{(k)}\Big(\sum_{\tau\in\mathcal{T}_k(G)}A_\tau^{(k)}\Big)
=\pi^{(k)}(D^{(k)})=0.\]
In other words, $A_\tau^{(k-1)}=\{0\}$ for all $\tau\in \mathcal{T}_{k}(G)$.
We have thus established conditions (a) and (b) in the statement. Condition~(c) follows from taking $k=1$ in condition~(b), since in this 
case we have $A^{(1)}_g=\{0\}$ for all $g\in G\setminus\{1\}$, which is
equivalent to $\alpha^{(1)}$ having the 1-decomposition property 
by item~(1) in~\autoref{eg:IntProp}. This concludes the proof.
\end{proof}

It follows that crossed products
of partial actions of finite groups can be
written as iterated extensions of crossed products of actions with the 
decomposition property, which in turn can be explicitly computed using
\autoref{prop:CrossedProductTupleOrb}. This fact has several applications
to the structure of partial crossed products. 
We summarize a relatively direct consequence in the following corollary,
while more involved applications are presented
in~\cite{AbaGarGef_partial_2020}. 

\begin{cor}\label{cor:FiniteGpCParbitraryPreserve}
Let $G$ be a finite group, and let 
$\mathbf{P}$ be a property for \ca s which 
passes to ideals, quotients, extensions, is
stable under tensoring with matrix algebras, and is preserved
by formation of crossed products by
global actions of $G$. Then $\mathbf{P}$ is 
preserved by formation of crossed products by
\emph{partial} actions of $G$.
This applies, in particular, to the properties
of being nuclear, having finite stable rank, or being
of type I.
\end{cor}
\begin{proof}
Let $(D^{(k)},\delta^{(k)})$ and $(A^{(k)}, \alpha^{(k)})$ be partial dynamical systems as 
in the conclusion
of \autoref{thm:ExtensionIntersProp}. 
For $k=2,\ldots,|G|$, we apply crossed products
to get the extension
\[\label{eqn:6.1}\tag{6.1}
\xymatrix{0\ar[r] & D^{(k)}\rtimes_{\delta^{(k)}}G \ar[r]& A^{(k)}\rtimes_{\alpha^{(k)}}G \ar[r] & A^{(k-1)}\rtimes_{\alpha^{(k-1)}}G\ar[r]&0.}
\]
Note that $A^{(k)}$ satisfies $\mathbf{P}$ for all 
$k=1,\ldots,|G|$, being a quotient of $A=A^{(|G|)}$,
and that the same is true for $D^{(k)}$ for 
$k=2,\ldots,|G|$, being an ideal in $A^{(k)}$.

We will show by induction that $A^{(k)}\rtimes_{\alpha^{(k)}}G$ satisfies $\mathbf{P}$.
For $k=1$, we have
$A^{(1)}\rtimes_{\alpha^{(1)}}G=A^{(1)}$,
which satisfies
$\mathbf{P}$.  
Assume now that we have proved the claim 
for $k-1$, and let us prove
it for $k$. 
Since \textbf{P} passes to extensions, the exact sequence
in ($\ref{eqn:6.1}$) 
implies that it suffices to show that
$D^{(k)}\rtimes_{\delta^{(k)}} G$ satisfies \textbf{P}.
Since $\delta^{(k)}$ has the $k$-decomposition
property, it follows from \autoref{prop:CrossedProductTupleOrb} that 
$D^{(k)}\rtimes_{\delta^{(k)}} G$ is isomorphic to a finite
direct sum of algebras of the form 
$M_{m_\tau}(D_\tau^{(k)}\rtimes_{\delta_\tau^{(k)}} H_\tau)$, for $\tau\in\mathcal{T}_k(G)$, where $D^{(k)}_\tau$
is an ideal in $D^{(k)}$, $H_\tau$ is a subgroup of $G$, and 
$\delta_\tau^{(k)}$ is the \emph{global} action obtained as the 
restriction of $\delta^{(k)}$ to $H_\tau$ 
and to $D^{(k)}_\tau$. This proves the claim, since $\mathbf{P}$ passes
to direct sums, is stable under tensoring with matrix algebras, and is preserved by formation of global crossed products, by assumption.

Since $A\rtimes_\alpha G$ equals 
$A^{(|G|)}\rtimes_{\alpha^{(|G|)}} G$, 
the claim shows the first assertion
in the statement.

Nuclearity, finiteness of the stable rank, and being
of type I
are known to satisfy the properties in the 
statement (see \cite[Sections~4--6]{Rie_sr} for permanence properties of stable rank and \cite[Theorem~2.4]{JeoOsaPhiTer_cancellation_2009} for 
preservation of the stable rank by global crossed products; see \cite[Theorem~5.6.2]{Mur_algebras_1990} and \cite[Theorem~4.1]{Rie_actions_1980} for similar results for type I).
\end{proof}

In the corollary above, preservation of nuclearity 
was already
known (and for arbitrary partial actions of amenable groups); see
\cite{Aba_enveloping_2003}, or \cite[Propositions~25.10 and 25.12]{Exe_partial_2017}. The result is
however new for stable rank and for type I.

The arguments used in
\autoref{cor:FiniteGpCParbitraryPreserve} cannot  
be applied to fixed point algebras to prove 
preservation results in this setting. Indeed,
it is not 
true that taking fixed point algebras of partial actions
of finite groups preserves short exact sequences,
as we briefly note in the following example:

\begin{eg}
Set $X=[-1,1]$ and $U=(-1,1)$, and let $\alpha$ be the 
partial action of $\Z_2$ on $A=C(X)$ induced by the 
homeomorphism $\sigma$ of $U$ given by $\sigma(x)=-x$
for all $x\in U$. Then $I=C_0(U)$ is an $\alpha$-invariant ideal in $A$, and the restriction of
$\alpha$ to it is the global action $\sigma^\ast$. The induced partial action of $\Z_2$ on $A/I\cong \C^2$ 
is the trivial partial 
action (see item~(1) in~\autoref{eg:IntProp}), so
we have $(A/I)^G=A/I$. Moreover, the quotient map 
$\pi\colon A\to A/I$ is given by $\pi(f)=(f(-1),f(1))$
for $f\in A$, and
the fixed point algebra $A^G$ is easily seen to be 
\[A^G=\{f\in C([-1,1])\colon f(x)=f(-x) \mbox{ for all } x\in [-1,1]\}.\] 
Thus $\pi(A^G)=\{(\lambda,\lambda)\in \C^2\colon \lambda\in \C\}$ does not equal $(A/I)^G=A/I$. It follows
that the equivariant short exact sequence
\[\xymatrix{0\ar[r]& (C_0(U),\sigma^\ast)\ar[r]& (C(X),\alpha)\ar[r] &(\C^2,\mbox{trivial})\ar[r]& 0}\]
does not induce a short exact sequence of fixed point
algebras.
\end{eg}

\section*{Acknowledgements}
The authors are grateful to the referee for
carefully reading the manuscript and 
making several helpful comments, and in particular
for pointing out that fixed point algebras of partial
actions of finite groups do not preserve exact sequences.

\end{document}